\documentclass[11pt]{amsart}
\usepackage{researchmacros}
\usepackage{ytableau}
\usepackage{mathtools}
\usepackage{enumitem}
\usepackage{tikz}
\usetikzlibrary{cd}
\usepackage[all]{xy}
\usepackage[vcentermath]{youngtab}
\usepackage{comment}

\usepackage{fullpage}

\usepackage{pifont}

\newtheorem{theorem}{Theorem}[section] 
\newtheorem{lemma}[theorem]{Lemma}     
\newtheorem{corollary}[theorem]{Corollary}

\theoremstyle{definition}
\newtheorem{definition}[theorem]{Definition}
\newtheorem{remark}[theorem]{Remark}

\title{$\Delta$--Springer varieties and Hall--Littlewood polynomials}
\author{Sean T. Griffin}

\address{Department of Mathematics, University of California Davis, Davis, CA, USA}
\email{stgriffin@ucdavis.edu}

\date{\today}

\newcommand{\st}{\,|\,}
\newcommand{\vspan}{\mathrm{span}}

\newcommand{\Hilb}{\mathrm{Hilb}}
\newcommand{\Frob}{\mathrm{Frob}}
\newcommand{\Frobq}[1]{\mathrm{Frob}(#1 ;q)}
\newcommand{\la}{\lambda}

\newcommand{\sort}{\mathrm{sort}}
\newcommand{\fl}{\mathrm{f\hspace{0.25pt}l}}

\newcommand{\bx}{\mathbf{x}}

\newcommand{\Fq}{\mathbb{F}_q}

\newcommand{\rev}{\mathrm{rev}}
\newcommand{\Comp}{\mathrm{Comp}}
\newcommand{\Par}{\mathrm{Par}}
\renewcommand{\Im}{\mathrm{Im}}
\newcommand{\bF}{\mathbb{F}}
\newcommand{\Stein}{\mathrm{Stein}}
\newcommand{\Spalt}{\mathrm{Spalt}}
\DeclareMathOperator{\inv}{\mathrm{inv}}
\DeclareMathOperator{\dinv}{\mathrm{dinv}}
\DeclareMathOperator{\PRD}{PRD}

\DeclareMathOperator{\coinv}{coinv}

\newcommand{\qbinom}{\genfrac{[}{]}{0pt}{}}

\begin{document}
\maketitle

\begin{abstract}
  The $\Delta$-Springer varieties are a generalization of Springer fibers introduced by Levinson, Woo, and the author that have connections to the Delta Conjecture from algebraic combinatorics. We prove a positive Hall--Littlewood expansion formula for the graded Frobenius characteristic of the cohomology ring of a $\Delta$-Springer variety. We do this by interpreting the Frobenius characteristic in terms of counting points over a finite field $\mathbb{F}_q$ and partitioning the $\Delta$-Springer variety into copies of Springer fibers crossed with affine spaces. As a special case, our proof method gives a geometric meaning to a formula of Haglund, Rhoades, and Shimozono for the Hall--Littlewood expansion of the symmetric function in the Delta Conjecture at $t=0$.
\end{abstract}

\section{Introduction}\label{sec:Introduction}

In~\cite{GLW}, Levinson, Woo, and the author introduced the $\Delta$-Springer varieties $Y_{n,\la,s}$  as a generalization of Springer fibers that give a compact geometric realization of the Delta Conjecture at $t=0$. Precisely, they showed that the symmetric function $\Delta'_{e_{k-1}}e_n|_{t=0}$ corresponds under the graded Frobenius characteristic map to the cohomology ring of a certain $\Delta$-Springer variety, up to a minor twist. In this article, we give a proof of the expansion, originally announced in~\cite{GriffinOSP}, of the graded Frobenius characteristic of the cohomology ring of a $\Delta$-Springer variety as a positive sum of modified Hall--Littlewood symmetric functions. As a special case, our proof gives a geometric explanation of the Hall--Littlewood expansion for $\Delta'_{e_{k-1}}e_n|_{t=0}$ given by Haglund, Rhoades, and Shimozono~\cite{HRS1}. 

\subsection{Background and context}

Springer fibers $\cB_\lambda$ are a family of subvarieties of the complete flag variety indexed by partitions $\lambda$ that have remarkable connections to representation theory. Notably, Springer~\cite{Springer-TrigSum,Springer-WeylGrpReps} constructed an action of the symmetric group $S_n$ on the cohomology ring of a Springer fiber and showed that the top nonzero cohomology group is an irreducible representation of $S_n$, and that all finite-dimensional irreducible representations of $S_n$ appear this way. Furthermore, the graded $S_n$-module structure of the cohomology ring is well studied. Under the Frobenius characteristic map that send an $S_n$-module to a symmetric function, the cohomology ring of a Springer fiber corresponds\footnote{Here, and throughout the paper, we define the grading on cohomology so that $H^{2k}$ is in degree $k$. This is possible since the odd degree cohomology groups of all of the varieties mentioned in this paper are $0$.} to the Hall--Littlewood symmetric function~\cite{Hotta-Springer},
\[
\widetilde H_\lambda(\bx;q) = \Frob(H^*(\mathcal{B}_\lambda;\bQ);q).
\]

Alternatively, the modified Hall--Littlewood symmetric functions can be obtained by counting $\Fq$-points. Letting $\Stein_\la^\mu$ be the Steinberg variety of partial flags of type $\mu$ that are preserved by a fixed nilpotent matrix of Jordan type $\lambda$, it is well known that for all prime powers $q$,
\begin{align}\label{eq:Springer-point-count}
\widetilde H_\lambda(\bx;q) = \sum_{\mu\vdash n} |\Stein_\la^\mu(\bF_q)|m_\mu(\bx),
\end{align}
where $m_\mu$ is the monomial symmetric function and $\Stein_\la^\mu(\bF_q)$ stands for the set of $\bF_q$ points of $\Stein_\la^\mu$.

Similar interpretations have been given for the modified Macdonald polynomials $\widetilde H_\lambda(\bx;q,t)$, which are a generalization of modified Hall-Littlewood symmetric functions with coefficients in two parameters $q$ and $t$. Haiman proved that $\widetilde H_\lambda(\bx;q,t)$ is the bi-graded Frobenius characteristic of the fiber of a certain tautological bundle on the Hilbert scheme of points in the plane~\cite{Haiman01,Haiman02}. On the other hand, Mellit realized $\widetilde H_\lambda(\bx;q,t)$ as a weighted count of points in an affine Springer fiber~\cite{Mellit}.

In a related line of work, the Delta Conjecture gives two combinatorial formulas for a symmetric function $\Delta'_{e_{k-1}}e_n$, where $\Delta'_{e_{k-1}}$ is a certain eigenoperator that diagonalizes the modified Macdonald polynomial basis. There has been recent progress on realizing this symmetric function naturally as the Frobenius characteristic of an $S_n$-module. Haglund, Rhoades, and Shimozono~\cite{HRS1} found a graded ring $R_{n,k}$ whose graded Frobenius characteristic is $\rev_q\circ \omega(\Delta'_{e_{k-1}}e_n|_{t=0})$, where $\omega$ is the involution that swaps elementary symmetric functions $e_n$ with complete homogeneous symmetric functions $h_n$, and $\rev_q$ reverses the coefficients as a polynomial in $q$. Pawlowski and Rhoades subsequently defined the space of spanning line configurations, which is a a smooth noncompact variety whose cohomology is $R_{n,k}$.

In~\cite{GriffinOSP}, the author generalized the ring $R_{n,k}$ to a family of quotient rings $R_{n,\lambda,s}$ that also include Tanisaki's presentation of the cohomology ring of a Springer fiber as a special case. Each ring $R_{n,\la,s}$ has the structure of a graded $S_n$-module whose graded Frobenius characteristic has several combinatorial formulas that extend previously-known formulas for $\Delta_{e_{k-1}}'e_n|_{t=0}$ and $\widetilde H_\la(\bx;q)$. In particular, there are two monomial symmetric function expansions in terms of \emph{inversions} and \emph{diagonal inversions}, respectively, of labeled diagrams called \emph{partial row-decreasing fillings}. We recall the relevant notation and statistics in Section~\ref{sec:Background}.
\begin{theorem}[\cite{GriffinOSP}]\label{thm:RInvFormula}
  We have
  \[
    \Frobq{R_{n,\la,s}} = \sum_{\mu\vdash n}\sum_{\varphi\in \PRD^\mu_{n,\la,s}} q^{\inv(\varphi)} m_\mu(\bx) = \sum_{\mu\vdash n}\sum_{\varphi\in\PRD^\mu_{n,\la,s}} q^{\dinv(\varphi)} m_\mu(\bx).
  \]
\end{theorem}
Levinson, Woo, and the author~\cite{GLW} then constructed a compact variety $Y_{n,\la,s}$, the $\Delta$-Springer variety, whose cohomology ring is $R_{n,\la,s}$. In the special case of $\lambda = (1^k)$ and $s=k$, the variety $Y_{n,(1^k),k}$ gives a compact geometric realization of the Delta Conjecture symmetric function at $t=0$, since $H^*(Y_{n,(1^k),k}) \cong R_{n,k}$. It was also shown that $Y_{n,\la,s}$ has many of the same geometric and representation-theoretic properties as Springer fibers, including a characterization of the $S_n$-module structure of the top cohomology group.

\begin{theorem}[\cite{GLW}]
  The variety $Y_{n,\la,s}$ is equidimensional of complex dimension
  \[
    d = n(\la) + (s-1)(n-k),
  \]
  where $n(\la) = \sum_i \binom{\la_i'}{2}$.
  For $s>\ell(\la)$, we have an isomorphism of $S_n$-modules
  \[
    H^{2d}(Y_{n,\la,s};\bQ) \cong \mathrm{Ind}\!\uparrow_{S_k\times S_{n-k}}^{S_n}(S^\la),
  \]
  where $S^\la$ is considered as a $S_k\times S_{n-k}$-module in which $S_{n-k}$ acts trivially.
\end{theorem}

\subsection{Results of this paper}
Our main theorem is a positive expansion of the graded Frobenius characteristic of $H^*(Y_{n,\la,s};\bQ)$ (equivalently of $R_{n,\la,s}$) into modified Hall--Littlewood symmetric functions, a result that was announced in~\cite{GriffinOSP}.
\begin{theorem}\label{thm:HL_exp_rev}
  We have 
  \begin{align}
    \Frob(H^*(Y_{n,\la,s};\bQ);q) &= \rev_q\left[ \sum_{\substack{\nu\in \Par(n,s),\\ \nu\supseteq \lambda}} q^{n(\nu/\lambda)} \prod_{i\geq 0} \qbinom{\nu_i'-\la_{i+1}'}{\nu_i'-\nu_{i+1}'}_q H_\nu(\bx;q)\right]\label{eq:rev-HL-formula}\\
    & = \sum_{\substack{\nu\in \Par(n,s),\\\nu\supseteq \la}} q^{\sum_i (s-\nu_i')(\nu_{i+1}'-\la_{i+1}')}\prod_{i\geq 0} \qbinom{\nu_i'-\la_{i+1}'}{\nu_i'-\nu_{i+1}'}_q \widetilde{H}_\nu(\bx;q),\label{eq:HL-formula}
  \end{align}
  where $\nu'_0\coloneqq s$.
\end{theorem}
A special case of this formula has led to the construction of \emph{higher Specht bases} for some infinite subfamilies of the rings $R_{n,\la,s}$~\cite{GillespieRhoades}.

The outline of the proof of Theorem~\ref{thm:HL_exp_rev} is as follows. We start with the fact that the left-hand side of~\eqref{eq:rev-HL-formula} is given by the inversion formula in Theorem~\ref{thm:RInvFormula}. We then show that $Y^\mu_{n,\la,s}$, the projection of $Y_{n,\la,s}$ down to a partial flag variety, has an affine paving such that the dimension of a cell is computed by the inversion statistic $\inv$. We use this to show that the $m_\mu(\bx)$ coefficient of the inversion formula in Theorem~\ref{thm:RInvFormula} can be computed by counting $\Fq$ points of $Y_{n,\la,s}^\mu$. We then show that the $\Fq$ points of $Y_{n,\la,s}^\mu$ can alternatively be computed by partitioning $Y_{n,\la,s}^\mu$ into a disjoint union of copies of affine spaces crossed with Steinberg varieties.

When $\lambda = (1^k)$ and $s=k$, the right-hand side of \eqref{eq:rev-HL-formula} is the right-hand side of \cite[Theorem 6.14]{HRS1}, the formula for $\Frob(R_{n,k};q)$ proven by Haglund, Rhoades, and Shimozono. Since $R_{n,k} \cong H^*(Y_{n,(1^k),k};\bQ)$ (up to doubling the degree), this can be seen as a solution to Problem 9.9 in \cite{Pawlowski-Rhoades}, but with the space of spanning line configurations $X_{n,k}$ replaced by the $\Delta$-Springer variety $Y_{n,(1^k),k}$.

\section{Background}\label{sec:Background}

\subsection{Compositions and partitions}

Let us recall basic notions about symmetric functions and their connections to $S_n$-modules. A \textbf{composition} $\alpha$ of $n$ of length $s$ is a tuple $\alpha=(\alpha_1,\alpha_2,\dots, \alpha_s)$ of nonnegative integers such that $\alpha_1+\cdots+\alpha_s = n$. We write $|\alpha|=n$ for the size of $\alpha$. Let $\Comp(n,s)$ be the set of compositions of $n$ with length $s$. We say $\alpha$ is a \textbf{strong composition} if $\alpha_i>0$ for all $i\leq s$. Given $\alpha$ a strong composition of $n$, let
\[
  \alpha[i] \coloneqq \{\alpha_1+\cdots+\alpha_{i-1}+1, \alpha_1+\cdots+\alpha_{i-1}+2,\dots, \alpha_1+\cdots+\alpha_{i}\}
\]
be the \textbf{$i$th block of $\alpha$}.

A \textbf{partition} is a composition $\lambda$ such that $\la_1\geq \la_2\geq \cdots \geq\la_s\geq 0$. The \textbf{length} of $\lambda$, written $\ell(\la)$, is the number of positive parts of $\la$. We sometimes write $\la\vdash n$ to mean that $\lambda$ is a partition of $n$. The \textbf{conjugate} partition $\la'$ is the partition of $n$ whose $i$th entry $\la_i'$ is the number of indices $j$ such that $\la_j\geq i$. Occasionally, we write $(a^b)$ to mean the partition $(a,a,\dots,a)$ with $b$ many $a$'s. We write $(\la,a)$ to mean the partition obtained by appending $a$ to the end of $\la$. Let $\Par(n,s)$ be the set of partititions $\la$ of $n$ such that $\ell(\la)\leq s$, where we identify partitions up to adding trailing $0$s.

Given $\alpha\in \Comp(n,s)$ and $\la\in \Par(n,s)$, we write $\alpha\supseteq \la$ if $\alpha_i\geq \la_i$ for all $i\leq s$. A \textbf{coinversion} of $\alpha$ is a pair $(i,j)$ such that $1\leq i<j\leq s$ and $\alpha_i<\alpha_j$. Let $\coinv(\alpha)$ be the number of coinversions of $\alpha$.

\subsection{Symmetric functions and q-analogues}\label{subsec:SymmetricFuncs}

A \textbf{symmetric function} $f$ is a formal power series in the variables $\bx \coloneqq \{x_1,x_2,\dots\}$ that is invariant under any permutation of the variables. Let $m_\lambda(\bx)$, $h_\lambda(\bx)$, and $s_\lambda(\bx)$ be the usual \textbf{monomial}, \textbf{complete homogeneous}, and \textbf{Schur symmetric functions}, each of which form a basis of the ring of symmetric functions as $\lambda$ ranges over all partitions of $n$. See \cite{Macdonald} for their definitions.
We define
\[
  \Omega[\bx] \coloneqq \prod_{i\geq 1} \dfrac{1}{1-x_i} =    \sum_{m\geq 0} h_m(\bx).
\]

To each finite-dimensional representation of $S_n$ over $\mathbb{Q}$, we associate a symmetric function via the Frobenius characteristic map as follows. For $\lambda\vdash n$, let $S^\lambda$ be the Specht module, which is the irreducible representation of $S_n$ indexed by $\lambda$. Given a finite-dimensional vector space $V$ over $\bQ$ that has the structure of an $S_n$-module, it decomposes as a direct sum of Specht modules, $V \cong \bigoplus_{\lambda\vdash n} (S^\lambda)^{\oplus a_\lambda}$, where $a_\lambda$ is the multiplicity of $S^\lambda$ in $V$. The \textbf{Frobenius characteristic} of $V$ is then defined to be
\[
  \Frob(V) \coloneqq \sum_{\lambda\vdash n} a_\lambda s_\lambda(\bx).
\]

Given a graded vector space $V = \bigoplus_{i\geq 0} V^{(i)}$ where $V^{(i)}$ is finite dimensional, its \textbf{Hilbert series} is the generating function
\[
  \Hilb(V;q) \coloneqq \sum_{i\geq 0} \dim(V^{(i)}) q^i.
\]
If each $V^{(i)}$ also has the structure of a finite-dimensional $S_n$-module, its \textbf{graded Frobenius characteristic} is
\[
  \Frob(V;q) \coloneqq \sum_{i\geq 0} \Frob(V^{(i)}) q^i.
\]

The ring $\Lambda_{\bQ(q)}$ has a basis given by the \textbf{Hall--Littlewood symmetric functions} $P_\mu(\bx;q)$, which have the property that
\[
  s_\lambda(\bx) = \sum_{\mu\vdash n} K_{\lambda,\mu}(q) P_\mu(\bx;q),
\]
where $K_{\lambda,\mu}(q)$ is the Kostka--Foulkes polynomial, see~\cite{Macdonald} for more details. The \textbf{dual Hall--Littlewood symmetric functions} are defined by
\[
  H_\mu(\bx;q) = \sum_{\lambda\vdash n} K_{\la,\mu}(q) s_\la(\bx).
\]
These symmetric functions are sometimes alternatively denoted by $Q'_\mu(\bx;q) = H_\mu(\bx;q)$. Given a polynomial $f(q) = a_0 + a_1q + \cdots +a_mq^m$ with symmetric function coefficients such that $a_m\neq 0$, the \textbf{$q$-reversal} of $f$ is $\rev_q(f) = a_m + a_{m-1}q + \cdots + a_1 q^{m-1} + a_0q^m$. In the case of Hall-Littlewood symmetric functions, the degree of $H_\mu(\bx;q)$ as a polynomial in $q$ is $n(\la) \coloneqq \sum_i\binom{\la_i'}{2}$. The reversal of $H_\mu(\bx;q)$ is the \textbf{modified (dual) Hall--Littlewood symmetric function}, written
\[
  \widetilde H_\mu(\bx;q) = \rev_q\left(H_\mu(\bx;q)\right) = q^{n(\mu)}H_\mu(\bx;q^{-1}).
\]

We use the following standard $q$-analogues of integers, factorials, and binomial coefficients.
\begin{align}
  [n]_q &\coloneqq 1+q+\cdots+q^{n-1},\\
  [n]_q! &\coloneqq [n]_q [n-1]_q\cdots [2]_q[1]_q,\\
  \qbinom{n}{k}_q&\coloneqq \frac{[n]_q!}{[k]_q![n-k]_q!}.
\end{align}

\subsection{Flag varieties and Schubert cells}\label{subsec:Flags}

Given a field $\bF$, a \textbf{partial flag} in $\bF^K$ is a nested sequence of vector subspaces of $\bF^K$,
\begin{align}
    V_\bullet = (V_1\subset V_2\subset\dots\subset V_m).
\end{align}
Given a strong composition $\mu = (\mu_1,\dots, \mu_m)$ such that $|\mu|=\mu_1+\cdots+\mu_m \leq K$,  define the \textbf{partial flag variety} to be the set of partial flags of $\bF^K$ such that the dimensions of the successive quotients $V_i/V_{i-1}$ are recorded by the parts of $\mu$,
\begin{align}
  \cB^\mu(\bF^K) \coloneqq \{V_\bullet = (V_1\subset\dots\subset V_m) \st  V_i\subseteq\bF^K,\, \dim(V_i/V_{i-1}) = \mu_i\text{ for }i\leq m\}.
\end{align}
Here, by convention we define $V_0 \coloneqq 0$. We occasionally write $\cB^\mu \coloneqq \cB^\mu(\bF^K)$ when $|\mu| = K$ and the field $\bF$ is clear from context.
The partial flag variety is realized as a projective algebraic variety as $G/P^\mu$, where $G=GL_K$ and $P^\mu$ is the parabolic subgroup of block upper triangular matrices with blocks of sizes $\mu_1,\mu_2,\dots, \mu_m,K-|\mu|$. We remark that although $\cB^\mu(\bF^K)$ is isomorphic to $\cB^{\mu,K-|\mu|}$, it will be notationally convenient for us to distinguish the two spaces.
  In the case when $K=n$ and $\mu = (1^n)$, $\cB^{(1^n)}$ is the \textbf{complete flag variety}.

  Returning to the general case when $|\mu|\leq K$, let $n \coloneqq |\mu|$, and let $[n] \coloneqq \{1,2,\dots, n\}$. Given an injective map $w : [n]\rightarrow [K]$, we say that $w$ is \textbf{$\mu$-increasing} if for each $i\leq \ell(\mu)$, $w$ is increasing on the $i$th block of $\mu$, that is $w_j<w_k$ for all elements $j<k$ of $\mu[i]$.
 Given any $w$, let $\sort_\mu(w)$ be the unique injective $\mu$-increasing function such that the images of the set $\mu[i]$ under $w$ and $\sort_\mu(w)$ are the same for each $i$.

Define $f_1,f_2,\dots,f_K$ to be the standard ordered basis of $\bF^K$.  Given a $\mu$-increasing injective map $w:[n]\rightarrow[K]$, let the {\bf coordinate flag} $F_\bullet^{(w)}\in \cB^\mu(\bF^K)$ be defined by setting $F^{(w)}_p=\vspan\{f_{w(1)},\ldots, f_{w(\mu_1+\cdots +\mu_p)}\}$ for all $p$ such that $1\leq p\leq \ell(\mu)$.  
Now define the {\bf Schubert cell} $C_w$ to be the $P^\mu$ orbit of $F^{(w)}_\bullet$. When $\bF = \bC$, the Schubert cells are the cells of an affine paving (in fact, a CW-complex) of the partial flag variety.

There is another description of the Schubert cells that will be useful for us. Given a vector $v = \sum_i a_if_i\in \bF^K$, we say its \textbf{leading term} is the term $a_if_i$ with $i$ maximal such that $a_i\neq 0$.  Given any $V_\bullet\in C_w\subseteq \cB^\mu(\bF^K)$, for each $i$ such that $1\leq i\leq \ell(\mu)$, there exist unique vectors $v_1,\dots,v_n\in \bF^K$ such that for $p\in \mu[i]$, we have $v_p\in V_i\setminus V_{i-1}$, $v_p$ has leading term $f_{w(p)}$, and
\begin{equation} \label{eq:SchubCellCoords}
v_{p} =f_{w(p)}+\sum_{h=1}^{w(p)-1} \beta_{w(p),h}\, f_h,
\end{equation}
for some $\beta_{w(p),h}$ such that $\beta_{w(p),h}=0$ if $h\in\{w(1),\ldots,w(p-1)\}$.  Note $V_i=\vspan\{v_1,\ldots,v_{\mu_1+\cdots+\mu_i}\}$.  The $\beta_{w(p),h}$ that are not required to be 0 can be taken as algebraically independent coordinates on $C_w$. We will say that \eqref{eq:SchubCellCoords} is the \textbf{standard coordinate representation} of $V_\bullet$ and that the $\beta_{w(p),h}$ are the \textbf{standard coordinates} of the Schubert cell. Equivalently, if $V_\bullet$ is represented as a matrix whose initial column spans are the parts of the flag $V_i$, then the $\beta_{w(p),h}$ are the matrix entries after column reducing.

Note that we have a projection map
\[\pi^\mu : \cB^{(1^n)}(\bF^K) \to \cB^\mu(\bF^K)\]
defined by sending $V_\bullet$ to $(V_{\mu_1},V_{\mu_1+\mu_2},\dots,V_n)$. Given $w$ a $\mu$-increasing injective function, it is evident from the coordinate description of the Schubert cell $C_w\subseteq \cB^{(1^n)}(\bF^K)$ that $C_w$ is mapped isomorphically under $\pi^\mu$ onto the corresponding Schubert cell $C_w\subseteq \cB^\mu(\bF^K)$ (which justifies our abuse of notation by not decorating $C_w$ with $\mu$).

\subsection{Springer fibers and Steinberg varieties}

Springer fibers are subvarieties of the complete flag variety studied by Springer~\cite{Springer-TrigSum,Springer-WeylGrpReps}. Given a nilpotent $n\times n$ matrix $N$ over $\bF$, the Springer fiber associated to it is
\[
\cB_N(\bF) \coloneqq \{V_\bullet \in \cB^{(1^n)} \st N V_i \subseteq V_i\text{ for }i\leq n\}
\]
with the reduced induced scheme structure.
Given two nilpotent matrices whose Jordan block sizes are recorded by the partition $\la\vdash n$, their associated Springer fibers are isomorphic, so we simply denote the Springer fiber by $\cB_\lambda \coloneqq \cB_N$, where $\la$ is the Jordan type of $N$. Springer showed that although the symmetric group $S_n$ does not act directly on $\cB_\lambda$, it does act on the cohomology ring $H^*(\cB_\lambda(\bC);\bQ)$ and used this action to geometrically construct the irreducible representations of $S_n$. The connection between Springer fibers and symmetric functions is summed up in the following elegant formula~\cite{Garsia-Procesi,Hotta-Springer},
\begin{align}\label{eqn:Hotta-Springer}
  \widetilde H_\lambda(\bx;q) = \Frobq{H^*(\cB_\lambda(\bC);\bQ)}.
\end{align}
(Recall our convention that we grade $H^*(\cB_\lambda(\bC);\bQ)$ so that the $2k$-th cohomology is in degree $k$.)

It can be shown by induction that the Springer fiber has an alternative definition as a variety where the conditions $NV_i\subseteq V_i$ are replaced with the conditions $NV_i\subseteq V_{i-1}$ for all $i$, since $N$ is nilpotent and we are considering complete flags. For partial flags of type $\mu$ a composition of $K$, these conditions give rise to two different varieties, which we refer to as the \textbf{Steinberg} and the \textbf{Spaltenstein varieties}, defined respectively as
\begin{align*}
  \Stein_\lambda^\mu(\bF) &\coloneqq \{V_\bullet\in \cB^\mu \st NV_i\subseteq V_i \text{ for }1\leq i\leq \ell(\mu)\},\\
  \Spalt_\lambda^\mu(\bF) &\coloneqq \{V_\bullet\in \cB^\mu \st NV_i \subseteq V_{i-1}\text{ for }1\leq i\leq \ell(\mu)\}.
\end{align*}
We note that $\Spalt_\lambda^\mu$ is denoted by $\cB^\lambda_\mu$ in~\cite{GLW}. By functoriality of the construction of the Springer action, there is an isomorphism of graded vector spaces ~\cite{Borho-Macpherson,Brundan-Ostrik}
\begin{align}\label{eq:BMIso}
  H^*(\Stein_\lambda^\mu(\bC);\bQ) \cong H^*(\cB_\lambda(\bC);\bQ)^{S_\mu},
\end{align}
where the superscript $S_\mu$ denotes taking the fixed subspace under the action of the Young subgroup $S_\mu\coloneqq S_{\mu[1]}\times\cdots \times S_{\mu[\ell(\mu)]}\subseteq S_n$ permuting the elements in the sets $\mu[1],\dots,\mu[\ell(\mu)]$ independently. Similarly, there is an identification of the cohomology of $\Spalt_\la^\mu(\bC)$ with the $S_\mu$-anti-invariants of the cohomology of the Springer fiber, up to a grading shift.

For $X$ a complex variety, an \textbf{affine paving} of $X$ is a filtration of $X$ by closed subvarieties
\[
  \emptyset=X_0\subseteq X_1\subseteq X_2\subseteq\cdots \subseteq X_m = X
\]
such that for all $i$, $X_i\setminus X_{i-1}\cong \bigsqcup_{j} C_{i,j}$, a disjoint union where $C_{i,j}\cong \bC^{a_{i,j}}$ for some $a_{i,j}$. 
If a compact complex variety $X$ has an affine paving, then the dimensions of the even degree cohomology groups can be computed by counting dimensions of cells,
\[
\dim_\bQ H^{2d}(X;\bQ) = \#\{(i,j) \st \dim_\bC(C_{i,j}) = d\},
\]
and the odd degree cohomology groups are $0$.


It is well known~\cite{Brundan-Ostrik,Fressepaving, Spaltenstein} that both $\Stein_\la^\mu(\bC)$ and $\Spalt_\la^\mu(\bC)$ have affine pavings. Hence, their odd degree cohomology groups are $0$.

\subsection{\texorpdfstring{$\Delta$}{Delta}-Springer varieties and affine pavings}\label{subsec:AffinePaving}
We now recall the definition of the $\Delta$-Springer variety $Y_{n,\la,s}$ and the affine pavings of $Y_{n,\la,s}$ constructed in~\cite{GLW}.

Fix an integer $n\geq 0$, a partition $\lambda$ of size $|\lambda|=k\leq n$, and an integer $s\geq \ell(\lambda)$. Define $\Lambda \coloneqq \Lambda(n,\la,s) = (\la_1+n-k,\la_2+n-k,\dots, \la_s+n-k)$, where $\la_i \coloneqq 0$ for $i>\ell(\la)$, and define $K \coloneqq |\Lambda| = k+(n-k)s$. Given $\bF$ a field and a nilpotent matrix $N_\Lambda$ over $\bF$ of Jordan type $\Lambda$, the \textbf{$\Delta$-Springer variety} over $\bF$ is defined to be
\[
Y_{n,\la,s}(\bF) \coloneqq \{V_\bullet \in \cB^{(1^n)}(\bF^K)\st NV_i\subseteq V_i \text{ for all }i,\, N_\Lambda^{n-k}\bF^{K}\subseteq V_n\}.
\]
The $\Delta$-Springer variety can equivalently be defined as the projection of a certain Spaltenstein variety. Letting $(1^n,(s-1)^{n-k}) = (1^n,s-1,s-1,\dots,s-1)$, where $s-1$ is repeated $n-k$ many times, and $\pi: \cB^{(1^n,(s-1)^{n-k})}\to \cB^{(1^n)}(\bF^K)$ be the projection map that forgets all but the first $n$ parts of the flag, we have ~\cite[Lemma 5.8]{GLW}
\begin{align}\label{eq:ProjectionDef}
  Y_{n,\la,s} = \pi\left(\Spalt^{(1^n,(s-1)^{n-k})}_{\Lambda}\right).
\end{align}
When $n=|\lambda| = k$ and $s$ is arbitrary, the $\Delta$-Springer variety specializes to $\cB_\la$, the usual Springer fiber.

We denote by $[\Lambda]$ the Young diagram of $\Lambda$ according to English convention, which formally is the set
\[
  [\Lambda] = \{(i,j) \st 1\leq i \leq \ell(\Lambda), 1\leq j\leq \Lambda_i\},
\]
where $(i,j)$ is the cell in the $i$th row from the top and the $j$th column from the left. There are two copies of $[\lambda]$ inside of $[\Lambda]$ that we will consider, which are respectively left justified and right justified inside of $[\Lambda]$,
\begin{align*}
  [\lambda] &\coloneqq \{(i,j) \st 1\leq i\leq \ell(\la), 1\leq j\leq \la_i\},\\
  [\lambda]^r &\coloneqq \{(i,j)\st 1\leq i\leq \ell(\la), \Lambda_i - \la_i + 1 \leq j\leq \Lambda_i\}.
\end{align*}

\begin{figure}
  \centering
  \includegraphics[scale=0.5]{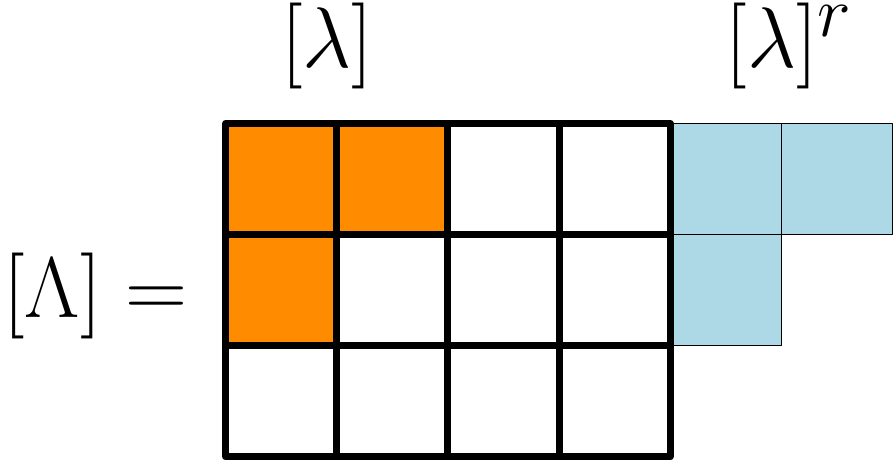}
  \caption{An example of $[\Lambda]$ for $n=7$, $\la=(2,1)$, $s=3$, with the two copies of the Young diagram of $\la$ shaded.\label{fig:Lambda}}
\end{figure}

The affine pavings of $Y_{n,\la,s} = Y_{n,\la,s}(\bC)$ are defined in terms of specific choices of the nilpotent matrix $N_\Lambda$. Recall $f_1,\dots, f_K\in\bC^K$ is the standard basis, and let $[K]\coloneqq \{1,2,\dots, K\}$. Given a bijection $T : [\Lambda] \to [K]$, define two nilpotent matrices $N_T$ and $N_T^t$ by
\begin{align}
  N_T (f_{T(i,j)}) &\coloneqq \begin{cases} 0 & \text{if } j = \Lambda_i\\ f_{T(i,j+1)} &\text{otherwise.}\end{cases}\\
  N_T^t(f_{T(i,j)}) &\coloneqq \begin{cases} 0 & \text{if }j=1 \\ f_{T(i,j-1)} &\text{otherwise.}\end{cases}
\end{align}
For example, for $T$ as in Figure~\ref{fig:SchubCompatExample}, we have $N_T f_5 = f_3$ and $N_T f_1 = 0$, whereas $N_{T}^t f_3 = f_5$ and $N_T^t f_8 = 0$.
Observe that both $N_T$ and $N_T^t$ have Jordan type $\Lambda$ by construction and that $N_T^t$ is simply the transpose of the matrix $N_T$. We define $Y_{N_T}$ and $Y_{N_T^t}$ to be the $\Delta$-Springer varieties for $N_T$ and $N_T^t$, respectively, where we are abusing notation by suppressing the data of $n$, $\la$, and $s$.

\begin{definition}\label{def:SchubCompat}
  We say that $T$ is \textbf{$(n,\la,s)$-Schubert compatible} if
\begin{itemize}
\item $T$ restricts to a bijection between $[\la]^r$ and $[k]$.
\item $T$ is decreasing along each row from left to right.
\item For all $(i,j)\in [\la]^r$, the label $T(i,j)$ is greater than all labels in column $j+1$.
\item For $i'<i$, we have $T(i',\Lambda_{i'}) < T(i,\Lambda_{i})$.
\item Whenever $T(a,b) > T(c,d)$ for $b,d>1$, then $T(a,b-1) > T(c,d-1)$.
\end{itemize}
When $n,\la$, and $s$ are clear from context, we simply say $T$ is \textbf{Schubert compatible}.
\end{definition}

\begin{definition}
  The \textbf{reading order} of $[\Lambda]$ is the sequence of cells obtained by reading down each column of $[\Lambda]$, ordering the columns from right to left. 
  The \textbf{reading order filling} of $[\Lambda]$ is the unique bijection $T:[\Lambda]\to [K]$ such that $T(i,j) = \ell$ if and only if $(i,j)$ is the $\ell$th cell in reading order. 
\end{definition}

It is noted in~\cite{GLW} that the reading order filling is Schubert compatible.
%
%
%
See the left-most filling in Figure~\ref{fig:SchubCompatExample} for an example of the reading order filling.

When $T$ is Schubert compatible, the intersection of a Schubert cell with $Y_{N_T}$ is either empty or a copy of affine space, and the nonempty intersections are the cells of an affine paving of $Y_{N_T}$. In fact, these cells have a recursive structure which we state next.

\begin{figure}
  \centering
  \includegraphics[scale=0.5]{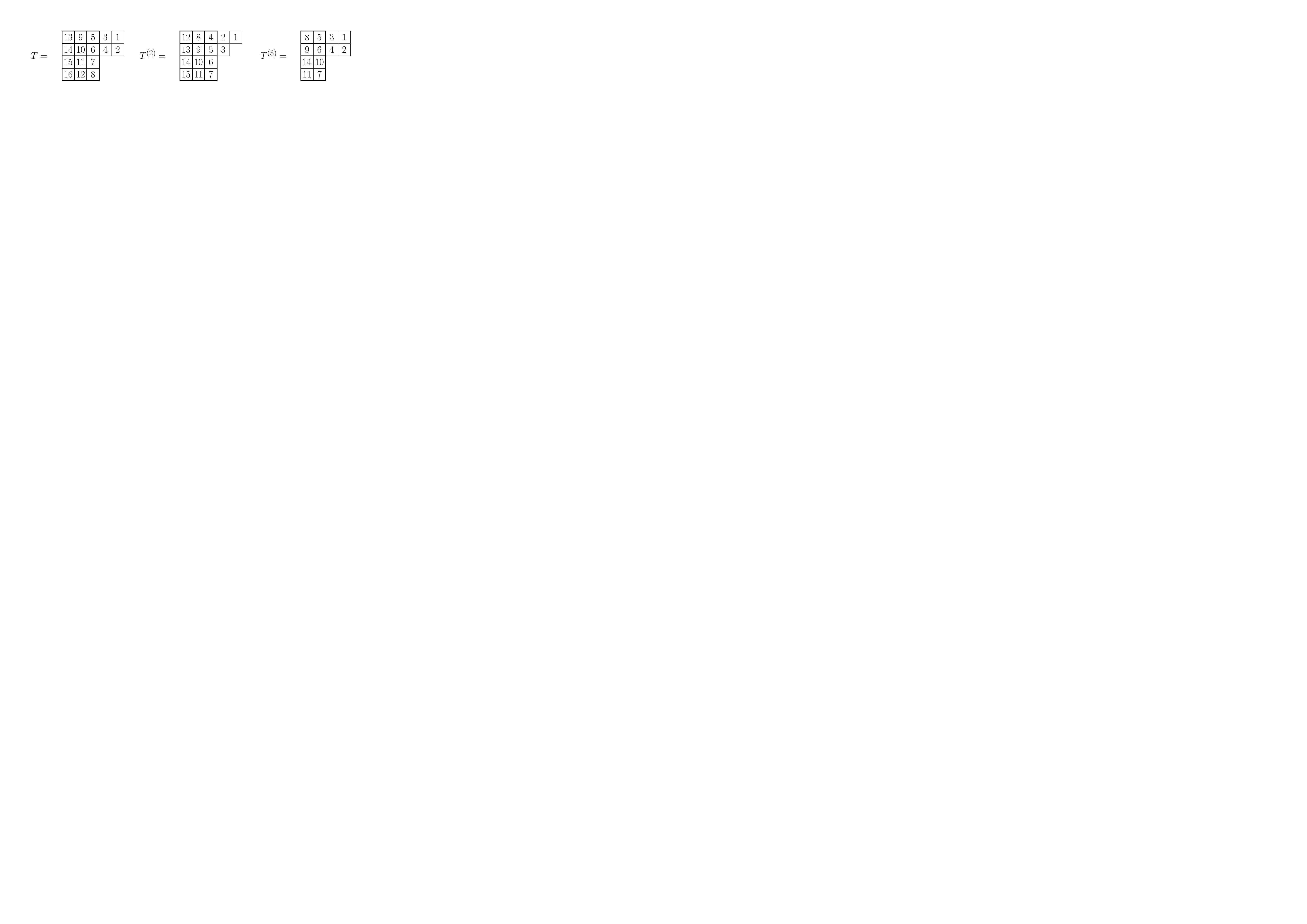}
  \caption{The reading order filling $T$ of $[\Lambda(7,(2,2),4)]$, and the associated $T^{(2)}$ and $T^{(3)}$.\label{fig:SchubCompatExample}}
\end{figure}  

Let $T$ be Schubert compatible. For $1\leq i \leq s$, define the \textbf{flattening function} $\fl_T^{(i)}$ and filling $T^{(i)}$ as follows. If $i\leq \ell(\la)$, then $\fl_T^{(i)}$ is the unique order-preserving function with the following domain and codomain, 
\[
  \fl_T^{(i)} : [K] \setminus \{T(i,\Lambda_i)\} \to [K-1].
\]
Let $T^{(i)}$ be the filling obtained by deleting the cell $(i,\Lambda_i)$, applying $\fl_T^{(i)}$ to the label in each remaining cell, and reordering the rows so that the labels in the right-most cells increase from top to bottom. 
If $i > \ell(\la)$, then $\fl_T^{(i)}$ is the unique order-preserving function
\[
  \fl_T^{(i)} : [K]\setminus (\{T(i,\Lambda_i)\}\cup \{T(i',1) \st i'\neq i\})\to [K-s].
\]
We define $T^{(i)}$ in the same way as the previous case, except we also delete the cells $(i',1)$ for $i'\neq i$ and shift those rows to the left by one unit before applying $\fl_T^{(i)}$ and reordering the rows. See Figure~\ref{fig:SchubCompatExample} for examples.

Given an injective map $w: [n]\to [K]$, we say $w$ is \textbf{admissible} with respect to $T$ if the image of $w$ contains $[k]$ and whenever $T(a,b) = w(i)$ for some $i$ then either $b=\Lambda_a$ or $T(a,b+1) = w(i')$ for some $i'<i$.

We have the following recursion for the cells $C_w\cap Y_{N_T}$.

\begin{lemma}[\cite{GLW}]\label{lem:CellRecursion}
  Let $T$ be Schubert compatible. 
  The intersection $C_w\cap Y_{N_T}$ is nonempty if and only if $w$ is admissible with respect to $T$. If $C_w\cap Y_{N_T}$ is nonempty, then there is an isomorphism
  \begin{equation}
C_w\cap Y_{N_T}\cong \bC^{i-1} \times (C_{\fl_T^{(i)}(w)} \cap Y_{N_{T^{(i)}}}).
 \end{equation}
\end{lemma}

\begin{remark}
  The proof of Lemma~\ref{lem:CellRecursion} given in~\cite{GLW} can be used without change to show that for any field $\bF$, there is a bijection
  \begin{equation}\label{eq:CellRecursionF}
    C_w\cap Y_{N_T}(\bF)\cong \bF^{i-1} \times (C_{\fl_T^{(i)}(w)} \cap Y_{N_{T^{(i)}}}(\bF)).
  \end{equation}
  We use the symbol $\cong$ throughout the paper to mean a bijection between sets, but all bijections below are easily seen to be isomorphisms of complex algebraic varieties in the case when $\bF=\bC$. Furthermore, the term `isomorphism' should be translated as `bijection' in the general setting and as `isomorphism of complex algebraic varieties' in the case $\bF=\bC$.
\end{remark}

A \textbf{partial row-decreasing filling} of $[\Lambda]$ is a filling of a subset of size $n$ of the cells of $[\Lambda] = [\Lambda(n,\la,s)]$ with positive integers such that the filled cells are right justified in each row, the labeling weakly decreases along each row, and each cell of $[\lambda]^r$ is filled.

Given $w$ admissible with respect to $T$, let $\PRD_T(w)$ be the partial row-decreasing filling of $[\Lambda]$ such that, for $1\leq i\leq n$, if $w(i) = T(a,b)$, then the cell $(a,b)$ of $[\Lambda]$ is labeled by $i$. It can be checked that the map $\PRD_T$ gives a bijection between admissible $w$ and partial row-decreasing fillings whose labels are  $1,2,\dots,n$ without repeats, hence these fillings also index the cells of $Y_{N_T}$. See Figure~\ref{fig:IPRD-bij} for an example of $\PRD_T(w)$ for $T$ the reading order filling of $[\Lambda(7,(2,2),4)]$ and $w = 2713594$ (where we have listed out the images of $1,2,\dots,7$ in order).

\begin{figure} 
  \centering
  \includegraphics[scale=0.5]{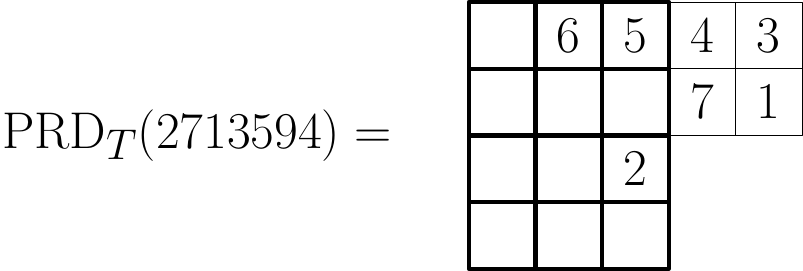} 
  \caption{For $T$ as in Figure~\ref{fig:SchubCompatExample}, the partial row-decreasing filling $\PRD_T(w)$ associated to the admissible function $w=2713594$.\label{fig:IPRD-bij}} 
\end{figure}

\subsection{Monomial symmetric function formulas}

In~\cite{GLW}, it is shown that the map on cohomology induced by inclusion of varieties
\[
  H^*(\cB^{(1^n)}(\bC^K);\bQ) \twoheadrightarrow H^*(Y_{n,\la,s}(\bC);\bQ)
\]
is surjective. There is a well-defined $S_n$ action on $H^*(Y_{n,\la,s}(\bC);\bQ)$ which is the unique one that makes the above map $S_n$-equivariant. Thus, $H^*(Y_{n,\la,s}(\bC);\bQ)$ has the structure of a graded $S_n$-module. In~\cite{GriffinOSP}, several formulas for the graded Frobenius characteristic of $H^*(Y_{n,\la,s}(\bC);\bQ)$ are given, which we recall next.

For $\mu\vdash n$, let $\PRD_{n,\la,s}^\mu$ be the set of partial row-decreasing fillings of $\Lambda = \Lambda(n,\la,s)$ with $\mu_i$ many $i$'s. Given such a labeling $\varphi\in \PRD_{n,\la,s}^\mu$, let $\varphi_{i,j}$ be the label in cell $(i,j)$. Given a pair of cells $((i,j)$, $(p,q))$ of $[\Lambda]$, we say they are an \textbf{attacking pair} if either $j=q$ and $i<p$, or if $j = p+1$ and $i > p$.

\begin{definition}
  Given $\varphi\in \PRD_{n,\la,s}^\mu$, a \textbf{diagonal inversion} of $\varphi$ is an attacking pair $((i,j),(p,q))$ of cells of $[\Lambda]$ such that one of the following hold,
  \begin{enumerate}
  \item [(D1)] $(i,j)$ and $(p,q)$ are filled such that $\varphi_{i,j} > \varphi_{p,q}$,
  \item[(D2)] $(i,j)$ is not filled and $(p,q)$ is filled.
  \end{enumerate}
  Let $\dinv(\varphi)$ be the number of diagonal inversions of $\varphi$.
\end{definition}

For $\varphi = \PRD_T(2713594)$ as in Figure~\ref{fig:IPRD-bij}, $\varphi$ has three diagonal inversions of type (D1); $((1,6),(2,6))$, $((2,5),(1,4))$, and $((1,4),(3,4))$; and three of type (D2); $((2,4),(3,4))$, $((4,4),(1,3))$, and $((2,4),(1,3))$, so $\dinv(\varphi) = 6$.

We have the following restatement of Theorem~\ref{thm:RInvFormula}.

\begin{theorem}[\cite{GLW}]\label{thm:GLWMainThm}
  We have
  \[
    \Frobq{H^*(Y_{n,\la,s}(\bC);\bQ)} = \sum_{\mu\vdash n} \sum_{\varphi\in \PRD_{n,\la,s}^\mu} q^{\dinv(\varphi)} m_\mu(\bx).
  \]
\end{theorem}

\begin{remark}
In Corollary~\ref{cor:CellDimInv}, we show that the statistic $\dinv$ counts the dimensions of the cells of $Y_{N_T}$ when $T$ is the reading order of $[\Lambda]$. There is another inversion statistic $\inv$ defined in~\cite{GriffinOSP,GLW} that also gives a monomial expansion of the graded Frobenius characteristic of $H^*(Y_{n,\la,s}(\bC);\bQ)$. We do not define it here, because it does not immediately come from a Schubert compatible filling of the cells of $[\Lambda]$. However, it may still be possible to show it counts dimensions of cells of $Y_{N_T}$ for some choice of filling $T$.
\end{remark}

\section{Counting $\Fq$-points of projected $\Delta$-Springer varieties}\label{sec:FqCounting}

In this section, we analyze the projections of the $\Delta$-Springer variety to other partial flag varieties. We then show that the graded Frobenius characteristic of $H^*(Y_{n,\la,s}(\bC);\bQ)$ can be written in terms of counting $\Fq$ points of these projected varieties.

Fix $n,\la,s,$ and $k=|\la|$ as in Subsection~\ref{subsec:AffinePaving}, let $\Lambda \coloneqq \Lambda(n,\la,s)$, and let $K \coloneqq |\Lambda| = k + (n-k)s = n+(n-k)(s-1)$.

\begin{definition}
  Let $\mu$ be a strong composition of $n$, and let $N_\Lambda$ is a nilpotent matrix of type $\Lambda$. Define the \textbf{projected $\Delta$-Springer variety} to be
  \begin{align}\label{eq:ProjectionOfDeltaSpringer}
  Y_{n,\la,s}^\mu(\bF) = \{V_\bullet \in \cB^\mu(\bF^K) \st N_\Lambda V_i\subseteq V_i \text{ for }i\leq \ell(\mu), \text{ and }N_\Lambda^{n-k}\bF^K\subseteq V_{\ell(\mu)}\}.
\end{align}
\end{definition}

\begin{lemma}
We have
\[
  Y_{n,\la,s}^\mu\coloneqq \pi^\mu\left(\Spalt_{\Lambda}^{(1^n,(s-1)^{n-k})}\right),
\]
where $\pi^\mu : \cB^{(1^n,(s-1)^{n-k})}\to \cB^\mu(\bF^K)$ is the usual projection map of partial flag varieties.
\end{lemma}

\begin{proof}
Let $\rho^\mu : \cB^{(1^n)}(\bF^K)\to \cB^\mu(\bF^K)$ be the usual projection map, which factors through $\pi^\mu$.  By~\eqref{eq:ProjectionDef}, it is immediate that $Y_{n,\la,s}^\mu = \rho^\mu(Y_{n,\la,s})$, which is contained in the right-hand side of \eqref{eq:ProjectionOfDeltaSpringer}. For the other containment, it is necessary to show that any $V_\bullet\in \cB^\mu(\bF^K)$ on the right-hand side can be extended to an element of $Y_{n,\la,s}$. This follows from the fact that
  \[
    (\rho^\mu)^{-1}(V_\bullet)\cap Y_{n,\la,s} \cong \cB_{N_1}\times \cdots \times\cB_{N_\ell},
  \]
  where $N_i$ is the nilpotent operator induced by $N_\Lambda$ on $V_i/V_{i-1}$. Indeed, since each of the Springer fibers in the product is nonempty, then $V_\bullet$ can be extended to an element $W_\bullet\in Y_{n,\la,s}$ so that $\rho^\mu (W_\bullet) = V_\bullet$. Thus, we have equality of the two sets.
\end{proof}

For $T:[\Lambda]\to [K]$ a bijection, we denote by $Y_{N_T}^\mu$ the projected $\Delta$-Springer variety $Y_{n,\la,s}^\mu$ for the specific choice of nilpotent $N_T$, and similarly for $N_T^t$. 

 \begin{lemma}\label{lem:ProjectionPaving}
 Let $T$ be Schubert compatible and $w$ admissible with respect to $T$.  The projection $\rho^\mu : Y_{N_T} \to Y_{N_T}^\mu$ maps $C_w\cap Y_{N_T}$ to $C_{\sort_\mu(w)}\cap Y_{N_T}^\mu$. When $w$ is $\mu$-increasing, $C_w\cap Y_{N_T}$ maps isomorphically onto $C_w\cap Y_{N_T}^\mu$.
Thus, when $\bF=\bC$ the subspaces $C_w \cap Y_{n,\la,s}^\mu(\bC)$ for $w$ admissible and $\mu$-increasing are the cells of an affine paving of $Y_{n,\la,s}^\mu(\bC)$.
\end{lemma}

\begin{proof}
  The first part of the lemma is immediate from the fact that $\rho^\mu$ maps $C_w$ to $C_{\sort_\mu(w)}$ and the definition of $Y_{N_T}^\mu$. Let $w$ be $\mu$-increasing. Since $\rho^\mu:\cB^{(1^n)}(\bF^K)\to \cB^{\mu}(\bF^K)$ maps $C_w$ isomorphically onto $C_w\subseteq \cB^\mu(\bF^K)$, then the restriction of $\rho^\mu$ to $C_w\cap Y_{N_T}$ maps $C_w\cap Y_{N_T}$ isomorphically onto its image, so it suffices to show that $C_w\cap Y_{N_T}$ maps surjectively onto $C_w\cap Y_{N_T}^\mu$.

  Let $V_\bullet\in C_w\cap Y_{N^T}^\mu$. For each $i$, let $v_1,\dots,v_n$ be the vectors such that for $p\in \mu[i]$, we have $v_p\in V_i\setminus V_{i-1}$, $v_p$ has leading term $f_{w(p)}$, and
\begin{equation}
v_{p} =f_{w(p)}+\sum_{h=1}^{w(p)-1} \beta_{w(p),h}\, f_h,
\end{equation}
for some $\beta_{w(p),h}$ such that $\beta_{w(p),h}=0$ if $h\in\{w(1),\ldots,w(p-1)\}$. We claim that the partial flag $W_\bullet\in \cB^{(1^n)}(\bF^K)$ defined by $W_i = \vspan\{v_1,\dots, v_i\}$ for $i\leq n$ is in $Y_{N_T}$. Indeed, since $N_T V_i\subseteq V_i$, then $N_TW_{\mu_i}\subseteq W_{\mu_i}$ for all $i$. For $\mu_1+\cdots+\mu_{i-1}+1 \leq p\leq \mu_1+\cdots+\mu_i$, then $N_Tv_p\in W_{\mu_i}$. Let $(a,b)$ be the coordinates of the label $w(p)$ in $T$, so $w(p) = T(a,b)$. Since $T$ is Schubert compatible, either $N_Tv_p=0$ or the leading term of $N_Tv_p$ is $N_Tf_{w(p)}=f_{T(a,b+1)}$. In the latter case, since $w$ is admissible with respect to $T$, then $T(a,b+1)=w(p')$ for some $p'<p$. Since $w$ is $\mu$-increasing we have $w(p')\notin \{p+1,\dots, \mu_i\}$, so the expansion of $N_Tv_p$ into the $v_i$ vectors cannot have any terms with nonzero coefficient on $v_{p+1},\dots, v_{\mu_i}$ (otherwise the leading term would not be $f_{w(p')}$), hence $N_Tv_p\in W_p$. Therefore, we have the containment $N_TW_p\subseteq W_p$ for all $p$.

Since $N_T^{n-k}\bF^{|\Lambda|} \subseteq V_{\ell(\mu)} = W_n$, then we have $W_\bullet\in Y_{n,\la,s}$ and by construction $\rho^\mu(W_\bullet) = V_\bullet$, so $\rho^\mu$ maps $C_w\cap Y_{n,\la,s}$ surjectively, hence isomorphically, onto $C_w\cap Y_{n,\la,s}^\mu$.

When $\bF=\bC$, since $C_w\cap Y_{n,\la,s}$ is isomorphic to affine space, then $C_w\cap Y_{n,\la,s}^\mu$ is as well. Since the $C_w$ are cells of an affine paving of $\cB^{\mu}(\bC^K)$, then the intersections $C_w\cap Y_{n,\la,s}^\mu(\bC)$ for $\mu$-increasing $w$ are the cells of an affine paving of $Y_{n,\la,s}^\mu(\bC)$.
\end{proof}

Recall that for $T$ Schubert compatible and $w$ admissible with respect to $T$, $\PRD_T(w)$ is defined to be the row-decreasing filling of $[\Lambda]$ such that if $w(i) = T(p,q)$, then the cell $(p,q)$ is labeled by $i$. Equivalently, if we let $\PRD_T(w)^{-1}(a)$ be the position of the label $a$ in $\PRD_T(w)$, we have
\begin{align}\label{eq:ShortCutLabel}
  w(a) = T(\PRD_T(w)^{-1}(a)).
\end{align}

\begin{definition}
  For $T$ Schubert compatible and $w$ admissible with respect to $T$, an \textbf{inversion of $w$ with respect to $T$} is a pair $(c,i)$ with $1\leq c\leq n$, $1\leq i\leq s$ such that one of the following conditions holds:
  \begin{itemize}
  \item[(IT1)] There exists a label $\ell>c$ of $\PRD_T(w)$ in row $i$ such that $w(\ell) < w(c)$,
  \item[(IT2)] The condition (IT1) does not hold, and there exists an empty cell $(i,j)$ of $\PRD_T(w)$ in row $i$ such that $T(i,j) < w(c)$.
  \end{itemize}
  Define $\inv_T(w)$ to be the number of such pairs.
\end{definition}

\begin{lemma}\label{lem:CellDim}
  Let $T$ be a Schubert-compatible filling of $[\Lambda]$. If $w$ is admissible with respect to $T$ and $\mu$-increasing, then
  \[
    C_w \cap Y_{N_T}^\mu \cong \bF^{\inv_T(w)}.
  \]
\end{lemma}

\begin{proof}
By Lemma~\ref{lem:ProjectionPaving}, it suffices to prove the result for the case $\mu = (1^n)$.
  We proceed by induction on $n$. In the base case when $n=0$ (and $s$ is arbitrary), then $\la= \emptyset$, and the unique admissible $w$ is the empty function $w=\emptyset$. Furthermore, $C_\emptyset\cap Y_{N_T} = Y_{N_T}$ is a single point and $\inv_T(w) = 0$. Thus, the base case holds.

Let $n\geq 1$ and $T$ and $w$ be arbitrary with $w(1) = T(i,\Lambda_i)$. To complete the induction, by Lemma~\ref{lem:CellRecursion} it suffices to show
  \begin{align}\label{eq:InvRecursion}
    \inv_T(w) - \inv_{T^{(i)}}\left(\fl_T^{(i)}(w)\right) = i-1.
  \end{align}
It can be checked using the definition of $\inv_T(w)$ that the left-hand side of \eqref{eq:InvRecursion} is the number of inversions of $w$ with respect to $T$ that are of the form $(1,i')$. We claim that $(1,i')$ is an inversion if and only if $i'<i$. Indeed, if $i'<i$, then $T(i',\Lambda_{i'}) < T(i,\Lambda_i)=w(1)$ since $T$ is Schubert compatible. If $(i',\Lambda_{i'})$ is filled in $\PRD_T(w)$, let its label be $\ell$. Since $w$ is injective, then $\ell>1$. Furthermore, $w(\ell) = T(\PRD_T(w)^{-1}(\ell)) = T(i',\Lambda_{i'}) < w(1)$, thus $(1,i')$ is an inversion of type (IT1). Otherwise, $(i',\Lambda_{i'})$ is not filled in $\PRD_T(w)$, so since $T(i',\Lambda_{i'}) < w(1)$, then $(1,i')$ is an inversion of type (IT2).

  Suppose $(1,i')$ is an inversion for some $i'\geq i$. Then in either of the cases (IT1) or (IT2), there exists a cell $(i',j')$ in row $i'$ such that $T(i',j') < w(1) = T(i,\Lambda_i)$ (where $T(i',j') = w(\ell)$ in the (IT1) case). By Schubert compatibility, we have $T(i,\Lambda_i) \leq T(i',\Lambda_{i'}) \leq T(i',j')$, a contradiction. Therefore, $(1,i')$ is an inversion if and only if $i'<i$, so the number of inversions counted by the left-hand side of \eqref{eq:InvRecursion} is $i-1$. The inductive step is thus complete.
\end{proof}

\begin{lemma}\label{lem:InvIsInv}
  Let $T$ be the reading order filling of $[\Lambda]$.  For all $w$ admissible with respect to $T$, we have $\inv_T(w) = \dinv(\PRD_T(w))$.
\end{lemma}

\begin{proof}
%
  Let $\varphi = \PRD_T(w)$. We define a map from diagonal inversions of $\varphi$ to inversions of $w$ with respect to $T$ as follows. For each diagonal inversion $((i,j),(p,q))$, let its corresponding inversion of $w$ be $(\varphi_{p,q},i)$. The pair $(\varphi_{p,q},i)$ is indeed an inversion: If $((i,j),(p,q))$ is a type (D1) inversion, then $(\varphi_{p,q},i)$ is a type (IT1) inversion, where $\ell=\varphi_{i,j}$. If $((i,j),(p,q))$ is type (D2), then $(\varphi_{p,q},i)$ is type (IT1) or type (IT2) depending on whether there exists a label $\ell>\varphi_{p,q}$ in row $i$ of $\varphi$ or not, respectively. We claim that the map defined above is a bijection between diagonal inversions of $\varphi$ and inversions of $w$ with respect to $T$.

  To show bijectivity, we give the inverse bijection, as follows. Given $(c,i)$ an inversion of $w$, let $(p,q)$ be the coordinates of the label $c$ in $\varphi$. If $(c,i)$ is of type (IT1), then there exists a label $\ell>c$ in row $i$ of $\varphi$ such that $w(\ell) < w(c)$. Since $w(\ell) < w(c)$, then either $i<p$ and $\ell$ is in a cell above and weakly to the right of $c$, or $i>p$ and $\ell$ is in a cell below and strictly to the right of $c$. In the first case, $(c,i)$ corresponds to the diagonal inversion $((i,q),(p,q))$. Indeed, the cell $(i,q)$ is either empty in $\varphi$, so $((i,q),(p,q))$ is a type (D2) diagonal inversion, or the cell $(i,q)$ is labeled and $\varphi_{i,q} \geq \ell > c$ be the row-decreasing property of $\varphi$ so $((i,q),(p,q))$ is a type (D1) diagonal inversion. In the second case when $i>p$, then $(c,i)$ corresponds to the diagonal inversion $((i,q+1),(p,q))$. Indeed, the cell $(i,q+1)$ is either empty in $\varphi$, so $((i,q+1),(p,q))$ is a type (D2) diagonal inversion, or the cell $(i,q+1)$ is labeled and $\varphi_{i,q+1}\geq \ell > c$ so $((i,q+1),(p,q))$ is a type (D1) diagonal inversion.

  Finally, if $(c,i)$ is of type (IT2), then there exists an empty cell $(i,j)$ in $\varphi$ such that $T(i,j) < w(c)$. Then either $i<p$ and $((i,q),(p,q))$ is a diagonal inversion of type (D2), or $i>p$ and $((i,q+1),(p,q))$ is a diagonal inversion of type (D2). It can then be checked that this is the inverse map, and hence the number of inversions of $w$ with respect to $T$ is equal to the number of diagonal inversions of $\varphi$.
\end{proof}

Combining Lemma~\ref{lem:InvIsInv} with Lemma~\ref{lem:CellDim}, we see that $\dinv$ has geometric meaning: it counts the dimensions of cells in $Y_{N_T}$.
\begin{corollary}\label{cor:CellDimInv}
  For $T$ the reading order filling of $[\Lambda]$ and $w$ that is $\mu$-increasing and admissible with respect to $T$, we have
  \begin{align}
    C_w\cap Y_{N_T}^\mu &\cong \bF^{\dinv(\PRD_T(w))}.
  \end{align}
\end{corollary}

We are now able to prove the main theorem of this section.

\begin{theorem}\label{thm:Delta-Springer-point-count}
For all prime powers $q$, we have
\[\Frob(H^*(Y_{n,\la,s}(\bC);\bQ);q) = \sum_{\mu\vdash n}|Y_{n,\la,s}^\mu(\Fq)|\, m_\mu(\bx).\]
\end{theorem}

\begin{proof}
  Let $T$ be the reading order filling of $[\Lambda]$, and let $\mu\vdash n$. We claim that
  \begin{align}\label{eq:mudinvsum}
    \sum_{\varphi \in \PRD_{n,\la,s}^\mu} q^{\dinv(\varphi)} = \sum_{\substack{w \text{ admissible wrt }T,\\ \mu\text{-increasing}}} q^{\dinv(\PRD_T(w))}.
  \end{align}
  Indeed, given $w$ admissible with respect to $T$ that is $\mu$-increasing, define $\varphi\in \PRD_{n,\la,s}^\mu$ by replacing the labels $\mu_1+\cdots+\mu_{i-1}+1, \mu_1+\cdots +\mu_{i-1}+2,\dots, \mu_1+\cdots +\mu_{i}$ in $\PRD_T(w)$ with $i$ for each $i\leq \ell(\mu)$. Since $w$ is $\mu$-increasing, then the labels $\mu_1+\cdots+\mu_{i-1}+1, \mu_1+\cdots +\mu_{i-1}+2,\dots, \mu_1+\cdots +\mu_{i}$ in $\PRD_T(w)$ are increasing in reading order, so there are no diagonal inversions among these labels. Therefore, it follows from the definition of $\dinv$ that $\dinv(\varphi) = \dinv(\PRD_T(w))$. 

Using \eqref{eq:mudinvsum}, we have  
  \begin{align}
    \Frob(H^*(Y_{n,\la,s}(\bC);\bQ);q)
    &= \sum_{\mu\vdash n} \sum_{\varphi \in \PRD_{n,\la,s}^\mu} q^{\dinv(\varphi)} m_\mu(\bx)\\
    &= \sum_{\mu\vdash n} \,\,\sum_{\substack{w \text{ admissible wrt }T,\\ \mu\text{-increasing}}} |C_w \cap Y_{N_T}^\mu(\Fq)| m_\mu(\bx)\\
    & = \sum_{\mu \vdash n} |Y_{N_T}^\mu(\Fq)| m_\mu(\bx).
  \end{align}
  The first equality is Theorem~\ref{thm:GLWMainThm}, the second follows from Corollary~\ref{cor:CellDimInv} and the fact that $\PRD_T$ restricts to a bijection between the set of $w$ that are admissible and $\mu$-increasing and $\PRD_{n,\la,s}^\mu$, and the third equality follows from the fact that the intersections $C_w\cap Y_{N_T}^\mu$ for $w$ that are admissible with respect to $T$ and $\mu$-increasing partition the space $Y_{N_T}^\mu$ by Lemma~\ref{lem:ProjectionPaving}.
\end{proof}

As a corollary, we obtain the following generalization of Borho and Macpherson's result \eqref{eq:BMIso}.

\begin{corollary}\label{cor:ProjectionInvariants}
  For all $\mu\vdash n$, we have an isomorphism of graded vector spaces
  \begin{align}
    H^*(Y_{n,\la,s}^\mu(\bC);\bQ) \cong H^*(Y_{n,\la,s}(\bC);\bQ)^{S_\mu}.
  \end{align}
\end{corollary}

\begin{proof}
  By Frobenius reciprocity (see e.g.~\cite{Brosnan-Chow}), the $m_\mu(\bx)$ coefficient of the graded Frobenius characteristic of a graded $S_n$-module is the Hilbert series of the $S_\mu$-invariant subspace. By Theorem~\ref{thm:Delta-Springer-point-count}, applying this fact to the cohomology of $Y_{n,\la,s}(\bC)$ yields 
  \begin{align}\label{eq:qCountHilbSmuinv}
    |Y_{n,\la,s}^\mu(\Fq)| = \Hilb(H^*(Y_{n,\la,s}(\bC);\bQ)^{S_{\mu}};q)
  \end{align}
  for all prime powers $q$.  By the last part of Lemma~\ref{lem:ProjectionPaving},
  \[
    \Hilb(H^*(Y_{n,\la,s}^\mu(\bC);\bQ);q) =  \sum_{\substack{w \text{ admissible wrt }T,\\ \mu\text{-increasing}}} q^{\dim_\bC(C_w \cap Y_{N_T}^\mu(\bC))}.
  \]
  By \eqref{eq:CellRecursionF}, it can be checked by induction that $|C_w\cap Y_{N_T}^\mu(\Fq)| = q^{\dim_\bC(C_w \cap Y_{N_T}^\mu(\bC))}$ for each $w$ admissible and $\mu$-increasing. Since the $C_w\cap Y_{N_T}^\mu(\Fq)$ partition $\Fq$, we have
  \begin{align}\label{eq:qCountHilbmu}
    |Y_{n,\la,s}^\mu(\Fq)| = \Hilb(H^*(Y_{n,\la,s}^\mu(\bC);\bQ);q).
  \end{align}
  Combining \eqref{eq:qCountHilbSmuinv} and \eqref{eq:qCountHilbmu}, the two spaces have the same Hilbert series, and are hence isomorphic as graded vector spaces.
\end{proof}

\section{A Springer fiber decomposition of $Y_{n,\la,s}$}\label{sec:SpringerDecomp}

In this section, we decompose the $\Delta$-Springer variety $Y_{n,\la,s}^\mu$ into subspaces that are isomorphic to Steinberg varieties crossed with affine spaces, so in particular $Y_{n,\la,s}$ decomposes into copies of Springer fibers crossed with affine spaces. We then use this to prove our expansion of the graded Frobenius characteristic of the cohomology ring in terms of Hall-Littlewood polynomials in Section~\ref{sec:ProofMainThm}.

\begin{definition}
The \textbf{reverse reading order} of $[\Lambda]$ is the sequence of cells obtained by reading down each column of $[\Lambda]$, ordering the columns from left to right. The \textbf{reverse reading order filling} of $[\Lambda]$ is the unique bijection $T:[\Lambda]\to [K]$ such that $T(i,j) = \ell$ if and only if $(i,j)$ is the $\ell$th cell in reverse reading order.
\end{definition}

Throughout this section, we let $T$ be the reverse reading order filling of $[\Lambda]$. For notational convenience, we use the simplified notation $N \coloneqq N_T$ and $N^t \coloneqq N_T^t$.

Given $\alpha\in \Comp(n,s)$ such that $\la\subseteq \alpha$, let $[\alpha]$ be the subdiagram of $[\Lambda]$ defined by
\[
  [\alpha] \coloneqq\{(i,j) : 1\leq i\leq s,\, 1\leq j\leq \alpha_i\}.
\]
Let $\bF^\alpha\subseteq \bF^K$, which is the subspace spanned by $f_{i}$ for $i$ a label of $T$ contained in $[\alpha]$. Similarly, let $\bF^\la \coloneqq (N^t)^{n-k}\,\bF^K$ be the subspace spanned by $f_i$ for $i$ a label of $T$ contained in $[\la]$.

Let $w: [n]\to [K]$ be the unique $(n)$-increasing injective function whose image is the set of labels of $T$ in $[\alpha]$, and let $C_{w}$ be the corresponding Schubert cell in the Grassmannian $\cB^{(n,K-n)} = \mathrm{Gr}(n,\bF^K)$.
Define
\begin{align}
  Z^\mu_\alpha &\coloneqq \{V_\bullet \in Y_{N^t}^\mu \st V_n = \bF^\alpha\},\\
  \widehat{Z}^\mu_\alpha & \coloneqq \{V_\bullet \in Y_{N^t}^\mu \st V_n\in C_{w}\}.
\end{align}
Note that since $N^t$ restricts to a nilpotent matrix on the subspace $\bF^\alpha$ with Jordan type $\sort(\alpha)$, then $Z_\alpha^\mu\cong \Stein_{\sort(\alpha)}^\mu$.

The following is the main theorem of this section. In the case $\mu= (1^n)$, it says that the subspaces $\widehat{Z}_{\alpha}$, which partition $Y_{n,\la,s}$, are isomorphic to a Springer fiber crossed with an affine space.

\begin{theorem}\label{thm:main-iso}
We have
\[
\widehat{Z}^\mu_\alpha \cong \bF^\ell\times Z_\alpha^\mu,
\]
where $\ell = \sum_i(s-\alpha_i')(\alpha_{i+1}'-\la_{i+1}') + \coinv(\alpha)$ and $\alpha_i'$ is the number of cells of $[\alpha]$ in the $i$th column.
\end{theorem}

In order to prove Theorem~\ref{thm:main-iso}, we need several definitions and lemmata.

\begin{definition}
  A pair $(i,j)$ with $i>j$ is a \textbf{free pair} for $\alpha$ if the cell with label $i$ of $T$ is in $[\alpha]/[\la]$ and the cell with label $j$ of $T$ is the leftmost cell of $[\Lambda]\setminus[\alpha]$ in its row.
\end{definition}

\begin{remark}
Let $w$ be the unique admissible $(n)$-increasing injective function whose image is the set of labels of $T$ in $[\alpha]$, as above. Free pairs $(i,j)$ are defined to correspond to a subset of standard coordinates $\beta_{i,j}$ of $C_w$. Note that $\beta_{i,j}$ is undefined unless $i$ is in the image of $w$, and that if $i$ is in the image of $w$, then $\beta_{i,j}=0$ if $j$ is a label in $[\alpha]$.
\end{remark}

\begin{lemma}\label{lem:number-free-pairs}
  The number of free pairs of $\alpha$ is
  \begin{align}\label{eq:number-free-pairs}
    \sum_i (s-\alpha_i')(\alpha_{i+1}'-\lambda_{i+1}') + \coinv(\alpha).
  \end{align}
\end{lemma}

\begin{proof}
  Let $(i,j)$ be a free pair for $\alpha$. If $i$ and $j$ are not in the same column, in which case $j$ must be in a column to the left of $i$, then associate $(i,j)$ to the pair $(i,j')$, where $j'$ is the label in the same row as $j$ and in the column immediately to the left of the column of $i$. This correspondence sets up a bijection between free pairs in which $i$ and $j$ are not in the same column and pairs $(i,j')$ where $i$ is a label in $[\alpha]\setminus [\lambda]$ and $j'$ is a label in $[\Lambda]\setminus [\alpha]$ in the column immediately to the left of $i$. Counting these latter pairs by the column containing $i$, we have the sum
  \[
    \sum_p (s-\alpha_p') (\alpha_{p+1}' - \lambda_{p+1}').
  \]

  If $i$ and $j$ are in the same column, then associate the free pair $(i,j)$ to the coinversion $(r,r')$ of $\alpha$, where $r$ is the row of $j$ and $r'$ is the row of $i$. Indeed, $(r,r')$ is a coinversion since $r<r'$ by the fact that $i>j$ and the definition of reverse inversion reading order, and $\alpha_r<\alpha_{r'}$ since $i$ is in $[\alpha]$ and $j$ is not. This correspondence is a bijection between free pairs with $i$ and $j$ in the same column and coinversions of $\alpha$. Thus, the total number of free pairs is given by \eqref{eq:number-free-pairs}.
\end{proof}

\begin{definition}
  For $(i,j)$ a free pair with $i>j$, let $U_{i,j}(t)$ be the $K\times K$ matrix such that
  \begin{align}
    U_{i,j}(t) (N^mf_i) &= N^m(f_i+tf_j),\label{eq:right-shift}\\
    U_{i,j}(t) ((N^t)^mf_i) &= (N^t)^m(f_i+tf_j),\label{eq:left-shift}
  \end{align}
  for all $m\geq 0$, and $U_{i,j}(t)f_\ell=f_\ell$ for all labels $\ell$ of $T$ that are not in the same row as $i$.

  Similarly, let $\widehat{U}_{i,j}(t)$ be the matrix such that
  \begin{align}
    \widehat{U}_{i,j}(t) (N^mf_i) &= N^m (f_i + tf_j)\label{eq:hat-left-shift}
  \end{align}
  for all $m\geq 0$, and $\widehat{U}_{i,j}(t) f_\ell=f_\ell$ for all labels $\ell$ of $T$ that are either not in the same row as $i$ or not weakly to the right of $i$.
\end{definition}

\begin{lemma}\label{lem:UpperTri}
The matrices $U_{i,j}(t)$ and $\widehat{U}_{i,j}(t)$ are unipotent upper triangular.
\end{lemma}

\begin{proof}
  By construction, $U_{i,j}(t)$ has $1$s along the diagonal, so it suffices to check it is upper triangular. Thus, it suffices to check that for all $m$, the vector $N^mf_i$ has index less than the vector $N^m f_j$ (if both are nonzero), and the vector $(N^t)^mf_i$ has index less than the vector $(N^t)^mf_j$ (if both are nonzero). Both of these claims follow immediately from the definition of free pairs and reverse inversion reading order. The same reasoning applies to $\widehat{U}_{i,j}(t)$.
\end{proof}

\begin{lemma}\label{lem:hat-equation}
  Let $(i_1,j_1),(i_2,j_2),\dots,(i_\ell,j_\ell)$ be the free pairs for $\alpha$ listed so that $i_1\geq i_2\geq \cdots \geq i_\ell$. Then for all $p\leq \ell$,
  \begin{align}\label{eq:hat-equation}
    U_{i_{p},j_{p}}(t_{p})\,U_{i_{p-1},j_{p-1}}(t_{p-1})\cdots U_{i_1,j_1}(t_1)\,\bF^\alpha = \widehat{U}_{i_{p},j_{p}}(t_{p})\,\widehat{U}_{i_{p-1},j_{p-1}}(t_{p-1})\cdots \widehat{U}_{i_1,j_1}(t_1)\,\bF^\alpha.
  \end{align}
\end{lemma}

\begin{proof}
  We proceed by induction on $p$. The base case $p=0$ is trivial, so suppose that for some $p\geq 0$, \eqref{eq:hat-equation} holds. It suffices to show that
  \begin{align}\label{eq:hat-equationplus}
    U_{i_{p+1},j_{p+1}}(t_{p+1})\widehat{U}_{i_{p},j_{p}}(t_{p})\cdots \widehat{U}_{i_1,j_1}(t_1)\bF^\alpha =  \widehat{U}_{i_{p+1},j_{p+1}}(t_{p+1})\widehat{U}_{i_{p},j_{p}}(t_{p})\cdots \widehat{U}_{i_1,j_1}(t_1)\bF^\alpha.
  \end{align}

Let $v\in \widehat{U}_{i_{p},j_{p}}(t_{p})\cdots \widehat{U}_{i_1,j_1}(t_1)\bF^\alpha$. Then
  \begin{align}\label{eq:LinearComb}
    U_{i_{p+1},j_{p+1}}(t_{p+1})v - \widehat{U}_{i_{p+1},j_{p+1}}(t_{p+1})v
  \end{align}
  is a linear combination of the vectors $(N^t)^mf_{j_{p+1}}$ for $m>0$. For a fixed $m>0$, let  $f_{j'} = (N^t)^mf_{j_{p+1}}$.
  Then $j'$ is in a column strictly to the left of $i_q$ for all $q\leq p+1$. Since the operators $\widehat{U}_{i_q,j_q}(t_q)$ for $q\leq p+1$ fix $f_{j'}$, and since $f_{j'}\in \bF^\alpha$ then $$f_{j'} \in \widehat{U}_{i_{p},j_{p}}(t_{p})\cdots \widehat{U}_{i_1,j_1}(t_1)\bF^\alpha.$$
  Furthermore, since $j'$ is not in the same row as $i_{p+1}$, then $U_{i_{p+1},j_{p+1}}(t_{p+1})$ and $\widehat{U}_{i_{p+1},j_{p+1}}(t_{p+1})$ fix $f_{j'}$, so $f_{j'}$ is in both sides of \eqref{eq:hat-equationplus}. Therefore, \eqref{eq:LinearComb} is in the intersection of the left-hand side and right-hand side of~\eqref{eq:hat-equationplus}. Since $U_{i_{p+1},j_{p+1}}(t_{p+1})v$ and $\widehat{U}_{i_{p+1},j_{p+1}}(t_{p+1})v$ are arbitrary elements of the left- and right-hand sides of \eqref{eq:hat-equationplus}, respectively, then the two sets are equal. The induction is thus complete.
%
\end{proof}

\begin{lemma}
  Let $(i_1,j_1),(i_2,j_2),\dots, (i_\ell,j_\ell)$ be the free pairs for $\alpha$, listed so that $i_1\geq i_2\geq \cdots \geq i_\ell$. Then we have a well-defined map
  \[
    \bF^\ell \times Z^\mu_\alpha \to \widehat{Z}^\mu_\alpha
  \]
  defined by sending $(\vec{t},V_\bullet)$ to ${U}_{i_\ell,j_\ell}(t_\ell)\cdots {U}_{i_2,j_2}(t_2) {U}_{i_1,j_1}(t_1)V_\bullet$.
\end{lemma}

\begin{proof}
  Letting $V_\bullet\in \widehat{Z}^\mu_\alpha$, it suffices to show that ${U}_{i_\ell,j_\ell}(t_\ell)\cdots {U}_{i_2,j_2}(t_2) {U}_{i_1,j_1}(t_1)V_\bullet \in \widehat{Z}^\mu_\alpha$. Since each ${U}_{i_p,j_p}(t_p)$ is unipotent upper triangular by Lemma~\ref{lem:UpperTri} and $V_n=\bF^\alpha \in C_{w}^{(n)}$, then
  \[
    U_{i_\ell,j_\ell}(t_\ell)\cdots U_{i_2,j_2}(t_2) U_{i_1,j_1}(t_1)V_n \in C_{w}^{(n)}.
  \]
  By contruction, we have $U_{i_p,j_p}(t_p) N^t = N^t U_{i_p,j_p}(t_p)$ for all $p\leq \ell$, so $N^t$ preserves each part of the partial flag $U_{i_\ell,j_\ell}(t_\ell)\cdots U_{i_2,j_2}(t_2) U_{i_1,j_1}(t_1)V_\bullet$.

  Finally, we must check that ${U}_{i_\ell,j_\ell}(t_\ell)\cdots {U}_{i_2,j_2}(t_2) {U}_{i_1,j_1}(t_1)V_n \supseteq \bF^\la$. Indeed, this follows from the fact that $V_n = \bF^\alpha\supseteq \bF^\la$, Lemma~\ref{lem:hat-equation}, and the fact that each $\widehat{U}_{i_p,j_p}(t_p)$ fixes $\bF^\la$. Thus, the map is well defined.
\end{proof}

\begin{lemma}\label{lem:ChangingEntries}
Let $V_\bullet\in Z^\mu_\alpha$, and let $(i_1,j_1),(i_2,j_2),\dots,(i_\ell,j_\ell)$ be the free pairs for $\alpha$ listed so that $i_1\geq i_2\geq \cdots \geq i_\ell$. For all $p\leq \ell$, then the standard coordinate $\beta_{i_p,j_p}$ of ${U}_{i_\ell,j_\ell}(t_\ell)\cdots {U}_{i_1,j_1}(t_1)V_n\in C_w^{(n)}$ is  $t_p$.
\end{lemma}

\begin{proof}
  By construction of $\widehat{U}_{i,j}(t)$, the standard coordinates of $\widehat{U}_{i,j}(t)V$ are obtained by applying $\widehat{U}_{i,j}(t)$ directly to each $v_p$ and collecting terms (in terms of matrix representatives, $\widehat{U}_{i,j}(t)$ sends column-reduced matrices to column-reduced matrices), hence the standard coordinates of ${U}_{i_\ell,j_\ell}(t_\ell)\cdots {U}_{i_1,j_1}(t_1)V_n = \widehat{U}_{i_\ell,j_\ell}(t_\ell)\cdots \widehat{U}_{i_1,j_1}(t_1)V_n$ are obtained by applying these operators to the standard coordinates of $V_n$. Since the $\beta_{i_p,j_p}$ coordinate of $V_n$ is $0$,  since $\widehat{U}_{i_p,j_p}(t_p)$ is the only $\widehat{U}_{i_q,j_q}$ operator that changes the $\beta_{i_p,j_p}$ coordinate, and since $\widehat{U}_{i_p,j_p}(t_p) f_{i_p} = f_{i_p} + t_pf_{j_p}$, then the $\beta_{i_p,j_p}$ coordinate after applying all of these operators is $t_p$.
\end{proof}

\begin{lemma}\label{lem:EqualEntries}
Let $V\in C_{w}^{(n)}$ such that $N^tV\subseteq V$, and let $(i,j)$ be a free pair for $\alpha$ such that for all labels $r<i$ of $T$ in $[\alpha]$ we have $f_{r}\in V$. Then for any $m>0$, if $N^mf_i = f_{i'}$ and $N^mf_j = f_{j'}$, then the $\beta_{i,j}$ and $\beta_{i',j'}$ coordinates of $V$ are equal.
\end{lemma}

\begin{proof}
Letting $v_q$ be the vectors for the standard coordinate representation of $V$, and letting $a_\ell = \beta_{i,\ell}$ and $b_{\ell'} = \beta_{i',\ell'}$ for notational convenience, we have
  \begin{align*}
    v_{w^{-1}(i)} &= f_i + a_j f_{j} + \sum_{j\neq \ell<i} a_\ell f_\ell,\\
    v_{w^{-1}(i')} &= f_{i'} + b_{j'} f_{j'} + \sum_{j'\neq \ell'<i'} b_{\ell'} f_{\ell'},\\
    (N^t)^m v_{w^{-1}(i')} &= f_i + b_{j'} f_j + \sum_{j'\neq \ell' < i'} b_{\ell'} ((N^t)^m f_{\ell'}).
  \end{align*}
  Since all of the above vectors are in $V$, then
  \begin{align}
    v_{w^{-1}(i)} - (N^t)^m v_{w^{-1}(i')} = (a_j-b_{j'}) f_j + \sum_{j\neq \ell<i} a_\ell f_\ell - \sum_{j'\neq \ell' < i'} b_{\ell'} ((N^t)^m f_{\ell'})\in V.\label{eq:DiffVect}
  \end{align}
  Observe that the $f_j$ coefficient of this vector is exactly $a_j-b_{j'}$, since the two sums contain no $f_j$ terms.

  Suppose by way of contradiction that $a_j\neq b_{j'}$. Then the vector~\eqref{eq:DiffVect} is nonzero, and since it is in $V$, its leading term must be a scalar multiple of $f_p$ for some $p$ a label in $[\alpha]$. Since $a_\ell = \beta_{i,\ell} = 0$ for all $\ell\in \{w(1),w(2),\dots, i\}$, then no nonzero terms in the first sum correspond to $\ell$ a label in $[\alpha]$. Therefore, the leading term must be $b_{\ell'}((N^t)^mf_{\ell'})$ for some $\ell'$ in the second sum. Let $\ell$ be such that $(N^t)^mf_{\ell'} = f_\ell$, and note that $\ell>j$ since it is the leading term. However, since $\ell'<i'$ by assumption, then $\ell<i$ (a property of reverse inversion reading order). Since we necessarily have that $\ell$ is a label in $[\alpha]$, then by the hypothesis in the statement of the lemma we have $f_\ell\in V$. Therefore, we can eliminate the $b_{\ell'}f_\ell = b_{\ell'}((N^t)^mf_{\ell'})$ term from \eqref{eq:DiffVect}. After repeating this process of eliminating terms from \eqref{eq:DiffVect}, we get a nonzero vector in $V$ whose leading term does not correspond to a label in $[\alpha]$, a contradiction. Therefore, we must have $a_j=b_{j'}$, which completes the proof.
\end{proof}

\begin{lemma}\label{lem:IsoSpringerProduct}
  Let $(i_1,j_1),(i_2,j_2),\dots, (i_\ell,j_\ell)$ be the free pairs for $\alpha$, listed so that $i_1\geq i_2\geq \cdots \geq i_\ell$. Then the map
  \[
    \bF^\ell \times Z^\mu_\alpha \to \widehat{Z}^\mu_\alpha
  \]
  defined by sending $(\vec{t},V_\bullet)$ to $U_{i_\ell,j_\ell}(t_\ell)\cdots U_{i_2,j_2}(t_2) U_{i_1,j_1}(t_1)V_\bullet$ is an isomorphism.
\end{lemma}

\begin{proof}
  First, observe that the map is well defined by Lemma~\ref{lem:UpperTri}. Second, by Lemma~\ref{lem:ChangingEntries} the $\beta_{i_m,j_m}$ coordinate of $U_{i_\ell,j_\ell}(t_\ell)\cdots U_{i_2,j_2}(t_2) U_{i_1,j_1}(t_1)V_\bullet$ is $t_m$ for each $m$. Therefore, any two pairs that map to the same partial flag must have the same $\vec{t}$ vector, and since the operators $U_{i,j}(t)$ are invertible, the map is injective.

  Finally, we show that the map is surjective. Let $V_\bullet\in \widehat{Z}^\mu_\alpha$, and let $t_m = \beta_{i_m,j_m}$ be the standard coordinate of $V_n$. We claim that $U_{i_1,j_1}^{-1}(t_1) U_{i_2,j_2}^{-1}(t_2)\cdots U_{i_\ell,j_\ell}^{-1}(t_\ell)V_\bullet \in Z^\mu_\alpha$. Observe that $U_{i,j}^{-1}(t) = U_{i,j}(-t)$. Thus, it is clear that this flag is in $\widehat{Z}^\mu_\alpha$, so it suffices to prove that
  \[
    U_{i_1,j_1}(-t_\ell)U_{i_2,j_2}(-t_2)\cdots U_{i_\ell,j_\ell}(-t_\ell)V_n = \bF^\alpha.
  \]
  This follows from the next claim.

  \medskip

  \textbf{Claim: }Letting $\beta_{*,*}$ be the standard coordinates of $U_{i_p,j_p}(-t_p)\cdots U_{i_\ell,j_\ell}(-t_\ell)V_n$, for all $1\leq p\leq q\leq \ell$ we have $\beta_{i_q,j_q}=0$, and for all pairs $(i_q,j)$ that are not free pairs we have $\beta_{i_q,j}=0$.

  \medskip

  We prove the claim by reverse induction on $p$. In the base case when $p=q=\ell$, then for all $r<i_\ell$ such that $r$ is a label of $T$ in $[\alpha]$, then $f_r\in V_n$, so we have $U_{i_\ell,j_\ell}(-t_\ell) V_n = \widehat{U}_{i_\ell,j_\ell}(-t_\ell)V_n$. Since the standard coordinates of $\widehat{U}_{i_\ell,j_\ell}(-t_\ell)V_n$ are obtained by applying $\widehat{U}_{i_\ell,j_\ell}(-t_\ell)$ to the standard coordinate representation of $V_n$, then $\widehat{U}_{i_\ell,j_\ell}(-t_\ell)$ eliminates the $\beta_{i_\ell,j_\ell}$ coordinate by construction, so the first part of the claim holds in the base case.

  If there is some pair $(i_\ell,j)$ with $j<i_\ell$ a label of $T$ in $[\Lambda]/[\alpha]$ that is not a free pair, let $j'$ be the label of the leftmost cell in the same row as $j$ that is not in $[\alpha]$ (which is distinct from $j$ by assumption). Let $i'$ be the label of the cell that is in the same row as $i_\ell$ and the same column as $j'$. Then $i'$ is not in $[\la]$ because $j'$ is not in $[\la]$, so $(i',j')$ is a free pair with $i'<i_\ell$, contradicting the fact the minimality of $i_\ell$. Therefore, the only pairs $(i_\ell,j)$ that are not free pairs are ones such that $j$ is a label of $[\alpha]$, so $\beta_{i_\ell,j} = 0$, and the base case is complete.

  In the inductive step, suppose the statement of the claim holds for $p$. Then by the inductive hypothesis, for all $r<i_{p-1}$ such that $r$ is a label of $T$ in $[\alpha]$ (note, this may exclude $i_p$ if $i_p=i_{p-1}$), then all $\beta_{r,j}$ coordinates of $U_{i_p,j_p}(-t_p)\cdots U_{i_\ell,j_\ell}(-t_\ell)V_n$ are $0$, so $f_r \in U_{i_p,j_p}(-t_p)\cdots U_{i_\ell,j_\ell}(-t_\ell)V_n$. Therefore,
  \[
    U_{i_{p-1},j_{p-1}}(-t_{p-1})U_{i_p,j_p}(-t_p) \cdots U_{i_\ell,j_\ell}(-t_\ell) V_n = \widehat{U}_{i_{p-1},j_{p-1}}(-t_{p-1})U_{i_p,j_p}(-t_p) \cdots U_{i_\ell,j_\ell}(-t_\ell)V_n.
  \]
  By construction, $\widehat{U}_{i_{p-1},j_{p-1}}(-t_{p-1})$ eliminates the $\beta_{i_{p-1},j_{p-1}}$ coordinate and does not change the coordinates guaranteed to be $0$ in the inductive hypothesis. Given a pair $(i_{p-1},j)$ such that $j$ is a label of $T$ not in $[\alpha]$, then there exist $m$ and $q>p-1$ such that $(N^t)^mf_{i_{p-1}} = f_{i_q}$ and $(N^t)^m f_{j_{p-1}} = f_{j_q}$. Note this implies $N^m f_{i_q} = f_{i_{p-1}}$ and $N^m f_{j_q} = f_{j_{p-1}}$. Since
  \[
    f_r \in U_{i_{p-1},j_{p-1}}(-t_{p-1})U_{i_p,j_p}(-t_p)\cdots U_{i_\ell,j_\ell}(-t_\ell)V_n
  \]
  for all $r\leq i_q$ such that $r$ is a label in $[\alpha]$, then by Lemma~\ref{lem:EqualEntries}, the $\beta_{i_q,j_q}$ and $\beta_{i_{p-1},j}$ coordinates of this flag are equal, and thus are both $0$. This completes the induction and the proof of the claim.

By the Claim, we have for $p=1$ that $U_{i_1,j_1}(-t_1)U_{i_2,j_2}(-t_2)\cdots U_{i_\ell,j_\ell}(-t_\ell)V_\bullet\in Z^\mu_\alpha$, so the map in the statement of the lemma is surjective. In the case when $\bF = \bC$, since the inverse is a regular map, it is an isomorphism of varieties.
\end{proof}

\begin{proof}[Proof of Theorem~\ref{thm:main-iso}]
    The isomorphism $\widehat{Z}^\mu_\alpha\cong \bF^\ell \times Z^\mu_\alpha$ follows by combining Lemma~\ref{lem:number-free-pairs} and Lemma~\ref{lem:IsoSpringerProduct}.
  \end{proof}

  \section{Proof of Theorem~\ref{thm:HL_exp_rev}}\label{sec:ProofMainThm}

  In this section, we use the results from Sections~\ref{sec:FqCounting} and \ref{sec:SpringerDecomp} to prove the main theorem, Theorem~\ref{thm:HL_exp_rev}. Throughout this section, we utilize the notation $\nu_0' \coloneqq s$ for $\nu\in \Par(n,s)$. Recall our convention that we consider $H^{2k}(Y_{n,\la,s};\bQ)$ to be the degree $k$ component of the cohomology ring.
  
  \begin{lemma}\label{lem:coinv-lemma}
    Let $\nu\in \Par(n,s)$ such that $\la\subseteq \nu$. Then
    \begin{align}\label{eq:CoinvSum}
      \sum_{\substack{\alpha\in \Comp(n,s),\\ \alpha\supseteq \la,\\ \sort(\alpha)=\nu}} q^{\coinv(\alpha)} = \prod_{i\geq 0} \qbinom{\nu_i' - \la_{i+1}'}{\nu_i' - \nu_{i+1}'}_q.
    \end{align}
  \end{lemma}

  \begin{proof}
    We proceed by induction on $m=\nu_1$. In the base case $m=\nu_1=1$, then $\la = (1^k)$ and $\nu = (1^n)$. In this case, the left-hand side can be rewritten as
    \[
      \sum_{\alpha \in \Comp(n-k,s-k)} q^{\coinv(\alpha)}=\qbinom{s-k}{n-k}_q = \qbinom{s-k}{s-n}_q = \qbinom{\nu_0'-\la_1'}{\nu_0'-\nu_1'}_q,
    \]
    which equals the right-hand side, so the base case holds.

    Let $m>1$, and assume by way of induction that the identity holds for all $\nu$ with $\nu_1=m-1$, and let $\nu\in \Par(n,s)$ such that $\la\subseteq \nu$. Define $\overline{\nu}$ and $\overline{\la}$ by removing all cells in column $m$ of the Young diagrams of $\nu$ and $\la$, respectively. Then $\alpha\in \Comp(n,s)$ such that $\la\subseteq \alpha$ and $\sort(\alpha) = \nu$ is obtained uniquely by the following process:
    \begin{itemize}
    \item Start with $\overline{\alpha} \in \Comp(n-\nu_m',s)$ such that $\overline{\la}\subseteq \overline{\alpha}$ and $\sort(\overline{\alpha}) = \overline{\nu}$,
    \item Append one cell to the first $\la_m'$ rows of $[\overline{\alpha}]$,
    \item Append one cell to $(\nu_m'-\la_m')$ many of the remaining rows of $[\overline{\alpha}]$ that have length $m-1$.
    \end{itemize}
    There are $(\nu_{m-1}'-\la_m')$ many rows of $[\overline{\alpha}]$ with length $m-1$ (apart from the first $\la_m'$ rows). Furthermore, the difference $\coinv(\alpha)-\coinv(\overline{\alpha})$ is the number of pairs $i<j$ such that $\alpha_i= m-1$ and $\alpha_j=m$. For a fixed $\overline{\alpha}$, the total contribution to the left-hand side of~\eqref{eq:CoinvSum} is thus
    \[
      q^{\coinv(\overline{\alpha})} \qbinom{\nu_{m-1}'-\la_m'}{\nu_m'-\la_m'}_q = q^{\coinv(\overline{\alpha})} \qbinom{\nu_{m-1}'-\la_m'}{\nu_{m-1}'-\nu_m'}_q,
    \]
    which is $q^{\coinv(\overline{\alpha})}$ times the $i=m-1$ factor on the right-hand side of \eqref{eq:CoinvSum}. An application of the inductive hypothesis then completes the proof.
  \end{proof}

\begin{theorem}\label{thm:HL-formula}
  We have
  \begin{align}\label{eq:HL-formula}
    \Frob(H^*(Y_{n,\la,s};\bQ);q) = \sum_{\substack{\nu\in \Par(n,s),\\\nu\supseteq \la}} q^{\sum_i (s-\nu_i')(\nu_{i+1}'-\la_{i+1}')}\prod_{i\geq 0} \qbinom{\nu_i'-\la_{i+1}'}{\nu_i'-\nu_{i+1}'}_q \widetilde{H}_\nu(\bx;q).
  \end{align}
\end{theorem}

\begin{proof}
  Combining \eqref{eq:Springer-point-count}, Theorem~\ref{thm:Delta-Springer-point-count}, Theorem~\ref{thm:main-iso}, and Lemma~\ref{lem:coinv-lemma}, we have for all prime powers $q$ the following string of identities.
  \begin{align}
    \Frob(H^*(Y_{n,\la,s};\bQ);q)  &= \sum_{\mu\vdash n} |Y^\mu_{n,\la,s}(\Fq)| m_\mu(\bx)\\
                                   &= \sum_{\mu\vdash n} \sum_{\substack{\alpha\in \Comp(n,s),\\\alpha\supseteq \la}} |\widehat{Z}^\mu_\alpha (\Fq)| m_\mu(\bx)\\
                                   &= \sum_{\mu\vdash n} \sum_{\substack{\alpha\in \Comp(n,s),\\\alpha\supseteq \la}} q^{\sum_i (s-\alpha_i')(\alpha_{i+1}' - \la_{i+1}') + \coinv(\alpha)} |Z^\mu_\alpha(\Fq)| m_\mu(\bx)\\
                                   &= \sum_{\substack{\alpha\in \Comp(n,s),\\\alpha\supseteq \la}} q^{\sum_i (s-\alpha_i')(\alpha_{i+1}' - \la_{i+1}') + \coinv(\alpha)} \left(\sum_{\mu\vdash n} |Z^\mu_\alpha (\Fq)| m_\mu(\bx)\right)\\
                                   &= \sum_{\substack{\alpha\in \Comp(n,s),\\\alpha\supseteq \la}} q^{\sum_i (s-\alpha_i')(\alpha_{i+1}' - \la_{i+1}') + \coinv(\alpha)} \widetilde{H}_{\sort(\alpha)}(\bx;q)\\
                                   &= \sum_{\substack{\nu\in \Par(n,s),\\\nu\supseteq \la}} q^{\sum_i (s-\alpha_i')(\alpha_{i+1}' - \la_{i+1}')} \prod_{i\geq 0}  \qbinom{\nu_i'-\la_{i+1}'}{\nu_i'-\nu_{i+1}'}_q \widetilde{H}_{\nu}(\bx;q).
  \end{align}
Since the identities hold for infinitely many values of $q$, the final identity holds for $q$ a formal parameter.
\end{proof}

\begin{corollary}\label{cor:rev-HL-formula}
  We have
  \begin{align}\label{eq:rev-HL-formula2}
    \Frob(H^*(Y_{n,\la,s};\bQ);q) = \rev_q\left[ \sum_{\substack{\nu\in \Par(n,s),\\ \nu\supseteq \lambda}} q^{n(\nu/\lambda)} \prod_{i\geq 0} \qbinom{\nu_i'-\la_{i+1}'}{\nu_i'-\nu_{i+1}'}_q H_\nu(\bx;q)\right].
  \end{align}
\end{corollary}

\begin{proof}
  By \cite[Theorem 1.2]{GLW}, the degree of $\Frob(H^*(Y_{n,\la,s};\bQ);q)$ is $n(\la)+(s-1)(n-k)$. By Theorem~\ref{thm:HL-formula}, it suffices to show that $\rev_q$ applied to the right-hand side of \eqref{eq:HL-formula} is equal to the expression in the right-hand side of \eqref{eq:rev-HL-formula2} inside of $\rev_q$. In other words, it suffices to prove that
  \begin{gather}
    q^{n(\la)+(s-1)(n-k)}\sum_{\substack{\nu\in \Par(n,s),\\\nu\supseteq \la}} q^{-\sum_i (s-\nu_i')(\nu_{i+1}'-\la_{i+1}')}\prod_{i\geq 0} \qbinom{\nu_i'-\la_{i+1}'}{\nu_i'-\nu_{i+1}'}_{q^{-1}} \widetilde{H}_\nu(\bx;q^{-1})
    =\\
    \sum_{\substack{\nu\in \mathrm{Par(n,s)},\\ \nu\supseteq \lambda}} q^{n(\nu/\lambda)} \prod_{i\geq 0} \qbinom{\nu_i'-\la_{i+1}'}{\nu_i'-\nu_{i+1}'}_q H_\nu(\bx;q).
  \end{gather}
  Since $\widetilde{H}_\nu(\bx;q^{-1}) =q^{-n(\nu)}H_\nu(\bx;q)$ and $\qbinom{a}{b}_{q^{-1}} = q^{b(a-b)}\qbinom{a}{b}_q$, we are reduced to proving that
  \begin{align}
    n(\la) + (s-1)(n-k) -\sum_i (s-\nu_i')(\nu_{i+1}'-\la_{i+1}') - \sum_i(\nu_i'-\nu_{i+1}')(\nu_{i+1}'-\la_{i+1}') -n(\nu) = n(\nu/\la),
  \end{align}
  which follows by a straightforward calculation using the facts that $\sum_i \nu_i' = n$ and $\sum_i\la_i' = k$. The proof is thus complete.
\end{proof}

\section{Further directions}

We list a few open problems and possible further connections with geometry and combinatorics:

\begin{itemize}
  \item Counting partial flags also has connections with the $q$-Burge correspondence by work of Karp and Thomas~\cite{KarpThomas}. Is there a generalization of the $q$-Burge correspondence that corresponds to counting pairs of points of the $\Delta$-Springer varieties?

\item The proof of Theorem~\ref{thm:HL_exp_rev} given in this article is geometric, relying on counting the points of the $\Delta$-Springer variety over $\bF_q$ using different ways to partition the space. A more direct combinatorial proof should be possible starting from Theorem~\ref{thm:RInvFormula}. Such a combinatorial proof would find a weight-preserving bijection that collects terms of this formula into sums that can be identified as monomial expansions of Hall-Littlewood symmetric functions. 

\item Formulas for Macdonald polynomials, which are two-parameter generalizations of Hall-Littlewood symmetric functions, have been obtained by Mellit~\cite{Mellit} as weighted sums over points of affine Springer fibers. It may be possible to define generalizations of $\Delta$-Springer varieties in the setting of affine flag varieties. One can then ask whether weighted point counts over these varieties yield symmetric functions that can simultaneously be defined using Macdonald eigenoperators. In particular, since $\Delta'_{e_{k-1}}e_n|_{t=0}$ is, up to a minor twist, the graded Frobenius characteristic of the cohomology of $Y_{n,(1^k),k}$, one might hope that extending the definition of $Y_{n,(1^k),k}$ to the affine setting and taking a weighted point count can recover the full symmetric function $\Delta'_{e_{k-1}}e_n$.
\end{itemize}

\section{Acknowledgements}

The author would like to thank Jake Levinson and Alexander Woo for many helpful discussions.

\bibliographystyle{hsiam}
\bibliography{Springer}

\end{document}